\documentclass[final, 12pt]{article}

\usepackage{amsmath}
\usepackage{amsfonts}
\usepackage{makeidx}
\usepackage{amsthm}
\usepackage{mathtools}
\usepackage{amssymb}
\usepackage{bm}
\usepackage{cite}
\usepackage{showkeys}
\usepackage[english]{babel}
\usepackage{dblaccnt}
\usepackage{accents}
\usepackage{graphicx}
\usepackage{subfig}
\usepackage{psfrag}
\usepackage{color}
\usepackage{float}
\usepackage{amscd}
\usepackage[mathscr]{euscript}
\usepackage{lipsum}
\usepackage{epstopdf}
\usepackage{algorithm}
\usepackage{algpseudocode}

\newcommand{\N}{\ensuremath{\mathbb{N}}}

\newcommand{\R}{\ensuremath{\mathbb{R}}}
\newcommand{\C}{\ensuremath{\mathbb{C}}}

\renewcommand{\d}{\mathrm{d}}
\newcommand{\dd}{\mathrm{d}}
\newcommand{\di}{\mathbf{d}}

\newcommand{\p}{\mathbf{p}}
\newcommand{\q}{\mathbf{q}}

\renewcommand{\sc}{\mathrm{sc}}

\renewcommand{\Im}{\mathrm{Im}\,}

\newcommand{\eps}{\varepsilon}

\renewcommand{\S}{\mathbb{S}}

\renewcommand{\H}{{\mathbf{H}}}

\newcommand{\E}{{\mathbf{E}}}
\newcommand{\Ei}{{\mathbf{E}_\mathrm{in}}}

\renewcommand{\div}{\mathrm{div}}
\newcommand{\curl}{\mathrm{curl}\,}

\newcommand{\x}{\mathbf{x}}
\newcommand{\y}{\mathbf{y}}
\newcommand{\z}{\mathbf{z}}

\renewcommand{\u}{\mathbf{u}}

\newtheorem{defi}{Definition}

\newtheorem{theorem}[defi]{Theorem}

\newtheorem{remark}[defi]{Remark}

\begin{document}

\title{Stable Reconstruction of Anisotropic Objects from Near-Field Electromagnetic Data}

\author{Tran H. Lan\thanks{Faculty of Applied Sciences, Ho Chi Minh City University of Technology and Education, Ho Chi Minh City, Vietnam; (\texttt{lanth@hcmute.edu.vn})} \and Dinh-Liem Nguyen\thanks{Department of Mathematics, Kansas State University, Manhattan, KS 66506, USA; (\texttt{dlnguyen@ksu.edu})}
}

\date{}
\maketitle

\begin{abstract}
This paper addresses   the electromagnetic inverse  scattering problem of 
determining the location and  shape of anisotropic  objects from near-field data. 
We investigate both cases involving the Helmholtz equation and Maxwell's equations for this inverse problem. 
Our study focuses on developing  efficient imaging functionals that enable a  fast and stable recovery of the anisotropic object.
The implementation of the imaging functionals is simple and avoids the need to solve an ill-posed problem.
 The resolution analysis    of the imaging functionals   is conducted using the Green representation formula. 
 Furthermore, we  establish stability estimates for these imaging functionals when  noise is present in the data.  
 To illustrate the effectiveness of the methods, we present numerical examples showcasing their performance.
\end{abstract}

\sloppy


\bigskip


\section{Introduction} 

We consider the inverse  scattering problem for both the Helmholtz equation and Maxwell's equations  concerning anisotropic objects. This inverse problem has important applications in various areas including nondestructive testing, radar imaging, medical imaging, and geophysical exploration~\cite{Colto2013}.
There has been a substantial body of literature  on both theoretical and numerical investigations of this inverse problem and its variants, see~\cite{Cakon2016, Colto2013} and references therein. For numerical solving  of  the inverse problem, sampling methods (e.g. linear sampling method~\cite{Colto1996}, factorization method~\cite{Kirsc1998}) offer distinct advantages  over nonlinear optimization-based techniques. These  sampling methods are fast, non-iterative and  do not require advanced a priori information about the unknown target. More details about sampling methods  can be found in~\cite{Potth2006, Kirsc2008, Colto2003,Cakon2006,   Cakon2011} and references therein.

The  work in this paper is inspired by the orthogonality sampling method (OSM) proposed in~\cite{Potth2010}. This method  is also studied in other works under the name of direct sampling method. While inheriting the advantages of the  aforementioned classical sampling methods (e.g. \cite{Colto1996, Kirsc1998}), the OSM is  attractive and promising thanks to its simplicity and efficiency. More precisely, the OSM is very robust against noise in the data, its stability is easily justified,  its implementation is simple and avoids the need to solve an ill-posed problem. However, the theoretical analysis of the OSM is far less developed compared with that of the classical sampling methods. The majority  of the published results  concerns the case of far-field measurements, see~\cite{Potth2010,Gries2011, Ito2012, Park2018, Leem2018, Harri2019, Ahn2020, Harris2022} for  some results on  the scalar Helmholtz equation  and~\cite{Ito2013, Nguye2019, Harri2020, Kang2021, Le2022} for results on the Maxwell's equations.

The results on the OSM concerning the case of near-field measurements are limited. The OSM studied in~\cite{Akinc2016} is specifically applicable to the 2D case with circular measurement boundaries. The 3D  case was studied in~\cite{Kang2021} under the small volume hypothesis of well-separated inhomogeneities. 
 The imaging functionals introduced in this paper,  along with their associated analysis, are not confined to small scatterers or  circular measurement boundaries. Nevertheless, they do require Cauchy data rather than solely relying on scattered field data. An imaging functional using near-field Cauchy data was studied in~\cite{Harri2022} for anisotropic objects. This functional is related to our imaging functional studied in the scalar case, but its resolution analysis  requires that  the material parameter of the scattering object must be smooth.
 Additionally, instead of using the Green's function as in~\cite{Harri2022} we only employ the imaginary part of the Green's function in the imaging functional. This approach helps  avoid dealing the blow-up singularity of the Green's function when the boundary of the sampling domain approaches the measurement boundary. 
 
  The use of near-field Cauchy data was also studied in~\cite{Le2023} in 
 the context of the isotropic Helmholtz equation with data generated by a single incident wave. This paper can be seen as an extended study of the findings in~\cite{Le2023} to the cases of the anisotropic Helmholtz equation and Maxwell's equations. 
 The latter cases are technically more complex than the isotropic Helmholtz equation case  demanding a meticulous analysis of the resolution and stability of the imaging functional. 
Additionally, it's worth noting that, for the case of Maxwell's equations, we suggest using several polarization vectors, denoted by $\p_n$, involved in the imaging functional that notably improve the  reconstruction results. Our resolution analysis of the imaging functional is based on   the integro-differential equation representation of the scattered field and the Green  representation formula of the imaginary part of the Green's function (or tensor) of the direct scattering problem. We show that the  imaging functional can be expressed as a functional that involves Bessel functions $J_0(k|y-z|)$ or $j_0(k|y-z|)$ and their derivatives where $z$ is a sampling point and $y$ is a point inside the unknown scatterer. These Bessel functions are crucial for  justifying the behavior of the imaging functional. Furthermore, we establish stability estimates for these imaging functionals 
when  noise is present in the near-field scattering data.
Finally, we illustrate the effectiveness of the methods with numerical examples in both 2D and 3D.

The rest of the paper is organized as follows. The second section includes the study of the Helmholtz equation case where we formulate the inverse problem of interest, introduce the imaging functional, analyze its behavior, and showcase numerical examples. Section 3 is dedicated to the case of Maxwell's equations with a structure that is similar to that of the scalar case.

\section{The Helmholtz Equation Case }
\label{sec2}

We first formulate the inverse problem  of interest.
Consider a penetrable inhomogeneous medium that 
occupies a bounded Lipschitz domain $D \subset \mathbb{R}^{n}$ ($n = 2$ or 3).  Assume that 
this medium is 
characterized by the bounded matrix-valued ($n\times n$ matrix) function $Q$ and that $Q = 0I_n$
in $\R^n \setminus \overline{D}$, where $I_n$ is the $n\times n$ identity matrix. Consider the incident plane wave
$$
u_{\mathrm{in}}(x,d) = e^{ikx\cdot d}, \quad  x \in \R^n, \quad d \in  \mathbb{S}^{n-1}: = \{x \in \R^n: |x| = 1\},
$$
where $k>0$ is the wave number and $d$ is the direction vector of propagation. We note that point sources can also be employed as incident waves without affecting the analysis of the reconstruction method studied in this paper.   We consider the following model problem for the scattering of  $u_{\mathrm{in}}(x,d)$ by the  inhomogeneous anisotropic medium
\begin{align}
\label{Helm}
& \div( (I_n + Q) \nabla u) + k^{2} u =0,\quad  x \in \mathbb{R%
}^{n},  \\
& u = u_\sc + u_\mathrm{in}, \\
 \label{radiation}
& \lim_{r\rightarrow \infty }r^{\frac{n-1}{2}}\left( \frac{\partial u_{\mathrm{sc}}}{
\partial r}-iku_{\mathrm{sc}}\right) =0,\quad r=|x|,
\end{align}
where $u(x,d)$ is the total field, $u_\sc(x,d)$ is the scattered field, and the Sommerfeld radiation condition~\eqref{radiation} holds
uniformly for all  directions $x/|x| \in \S^{n-1}$. 
We refer to~\cite{Kirsc2007b} for the well-posedness of this scattering problem. For the study of the inverse problem discussed below we 
 assume that the direct problems is well-posed.

Consider a regular domain $\Omega \subset \R^n$ such that $\overline{D} \subset \Omega$
and denote by $\nu(x)$ the outward normal unit vector to $\partial \Omega$ at $x$.

\textbf{Inverse problem.} 
Given   $u_\sc(\cdot,d_j)$  and $ \partial u_\sc(\cdot,d_j)/ \nu$ on $\partial \Omega$ for 
$d_j \in  \mathbb{S}^{n-1}$ where $j = 1, 2, \dots, N$, determine $D$.

We denote by $\Phi(x,y)$ the free-space Green's function of the scattering problem~\eqref{Helm}--\eqref{radiation}. It is well known that 
\begin{equation} 
\label{green}
 \Phi(x,y)= 
\begin{cases}
\frac{i}{4}H^{(1)}_0(k|x-y|), & \text{in } \R^2 , \\ 
\frac{e^{ik|x-y|}}{4\pi|x-y|}, & \text{in } \R^3.
\end{cases}
\end{equation}
 It is also  known from~\cite{Kirsc2007b} that  the scattered wave $u_\sc$ (weak solution in $H^1_{\mathrm{loc}}(\R^n)$) of 
 problem~\eqref{Helm}--\eqref{radiation} can be described by the following integro-differential   equation
\begin{align}
\label{LS}
u_\sc(x) = \div_x \int_D \Phi(x,y) Q(y) \nabla u(y)\dd y, \quad x \in \R^n.
\end{align}
This   integro-differential   equation formulation for the scattered wave $u_\sc$  is important to the study of the imaging functional in 
the next section.

\subsection{Stable Imaging Functional   }
\label{sec3}
In this section we introduce the  imaging  functional and analyze its properties.  
This functional was introduced in~\cite{Le2023} for the case of the isotropic Helmholtz equation 
with data generated by a single incident wave. For the case of data generated by multiple incident waves
we   define the  imaging  functional  $I(z)$ as
\begin{equation}
\label{I1}
I(z) := \sum_{j  = 1}^N \left| \int_{\partial \Omega} u_\sc(x,d_j)  \frac{\partial\Im \Phi(x,z)}{\partial \nu(x)} - \frac{\partial u_\sc(x,d_j)}{\partial \nu(x)}  \Im \Phi(x,z)\dd s(x)  \right|^p, 
\end{equation}
where  $p \in \N$ is used to sharpen the significant peaks of $I(z)$. 
 
Recall that $J_0$ and $j_0$ are respectively a Bessel function and a spherical Bessel function of the first kind.
The behavior of  $I(z)$ is analyzed in the following theorem.
\begin{theorem}
\label{theorem1}
The imaging  functional $I(z)$  satisfies
\begin{equation}
\label{thm1}
I(z) := \sum_{j  = 1}^N \left| \int_{D}  \mathcal{K}(y,z) \cdot Q(y) \nabla u(y,d_j)\, \dd y \right|^p,
\end{equation}
where 
\begin{equation*} 
\mathcal{K}(y,z) =
\begin{cases}
\frac{1}{4} \nabla_y J_0(k|y-z|), & \text{in } \R^2 , \\ 
\frac{k}{4\pi} \nabla_y j_0(k|y-z|), & \text{in } \R^3.
\end{cases}
\end{equation*}
Furthermore 
\begin{align}
\label{decayrate}
I(z) = O\left(\frac{1}{\mathrm{dist}(z,D)^{\frac{p(n-1)}{2}}} \right)\quad  \text{ as } \mathrm{dist}(z,D) \to \infty,
\end{align}
where $\mathrm{dist}(z,D)$ is the distance from $z$ to $D$. 
\end{theorem}
\begin{remark}
For a fixed $y$ in $\R^2$,  we observe that while $|\partial_{y_1}J_0(k|y-z|)|^2, |\partial_{y_2}J_0(k|y-z|)|^2$, and $|\nabla_yJ_0(k|y-z|)|^2$ do not exhibit peaks at $z = y$, 
they do have relatively larger magnitudes in a small neighborhood near $y$. As 
$z$ is away from $y$ these magnitudes decay quickly, as illustrated in Figure~\ref{fi0}. Especially the peaks near $y$ are sharper and stronger for higher values of $k$.  Thus for an extended object $D$, according to Theorem~\ref{theorem1},
we anticipate that the imaging functional $I(z)$ (with, e.g., $p = 2$) will attain higher values when $z$ belongs to $D$, while 
$I(z)$ will be considerably smaller when $z$  lies outside $D$. This is indeed confirmed in the numerical study.

\end{remark}

\begin{figure}[h!!!]
\centering
\subfloat[$|\partial_{y_1}J_0(8|y-z|)|^2$]{\includegraphics[width=5.5cm]{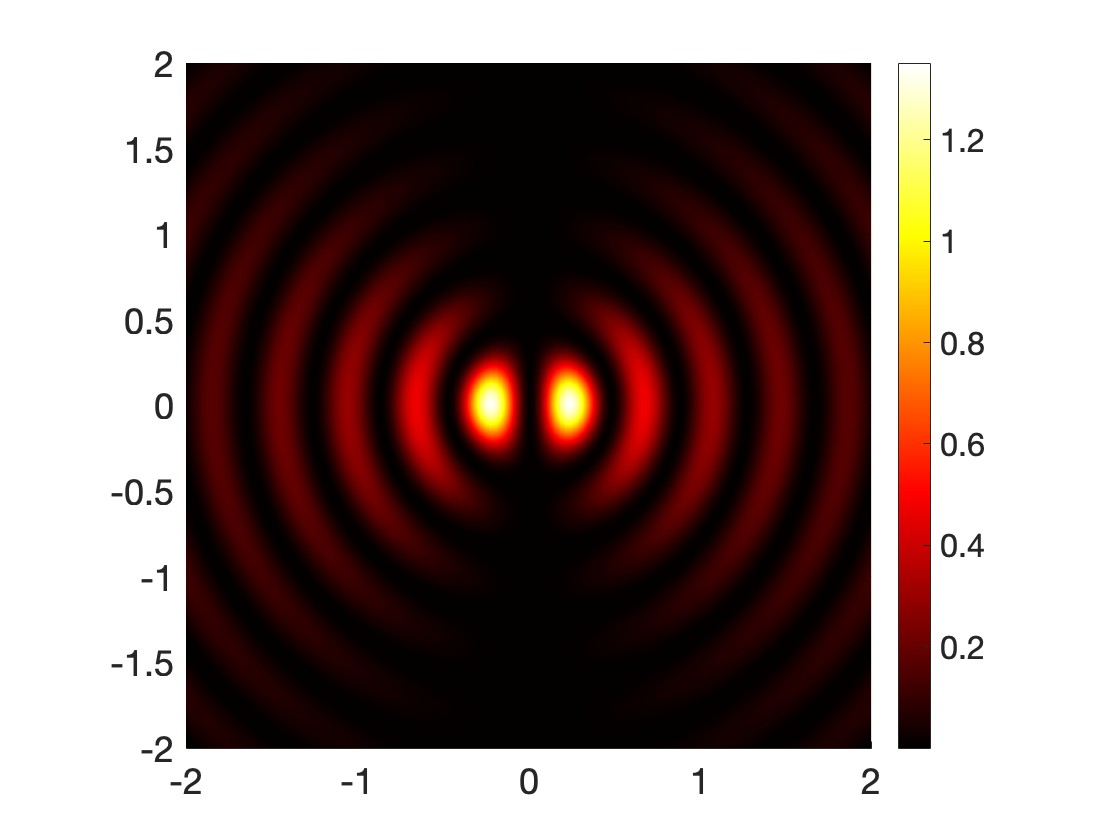}} \hspace{-0.7cm}
\subfloat[$|\partial_{y_2}J_0(8|y-z|)|^2$]{\includegraphics[width=5.5cm]{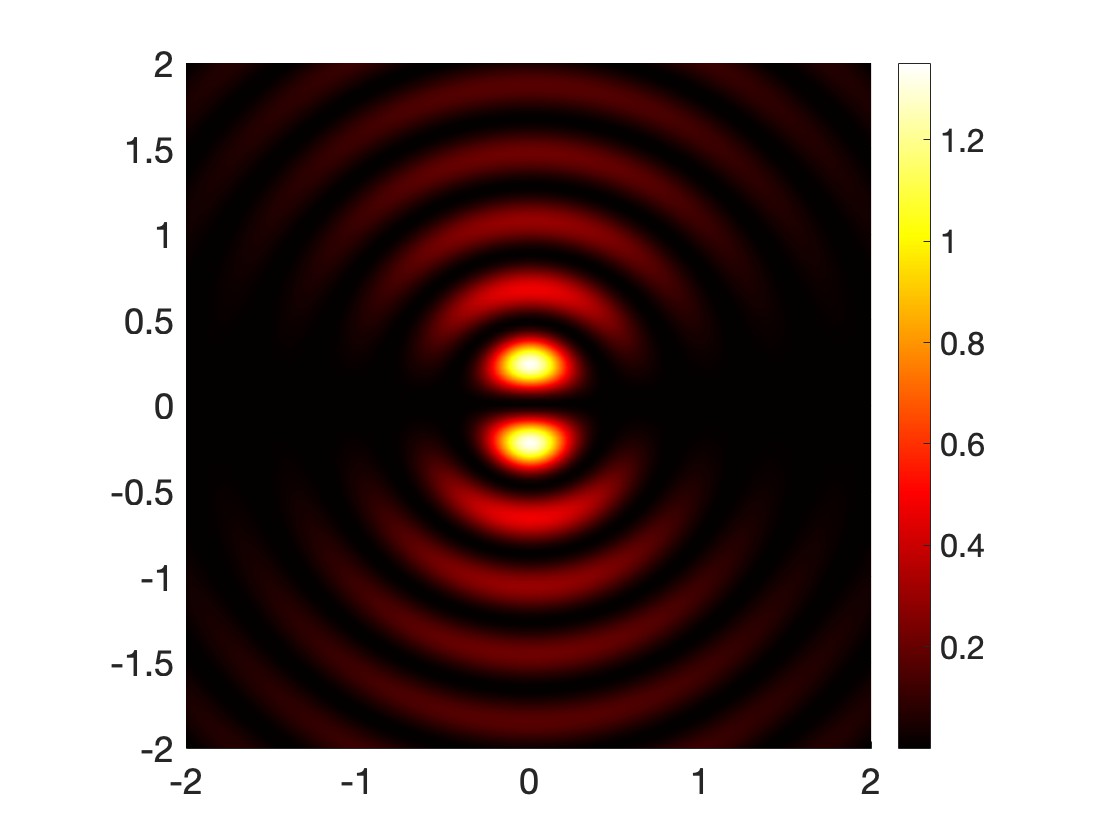}} \hspace{-0.7cm}
\subfloat[$|\nabla_yJ_0(8|y-z|)|^2$]{\includegraphics[width=5.5cm]{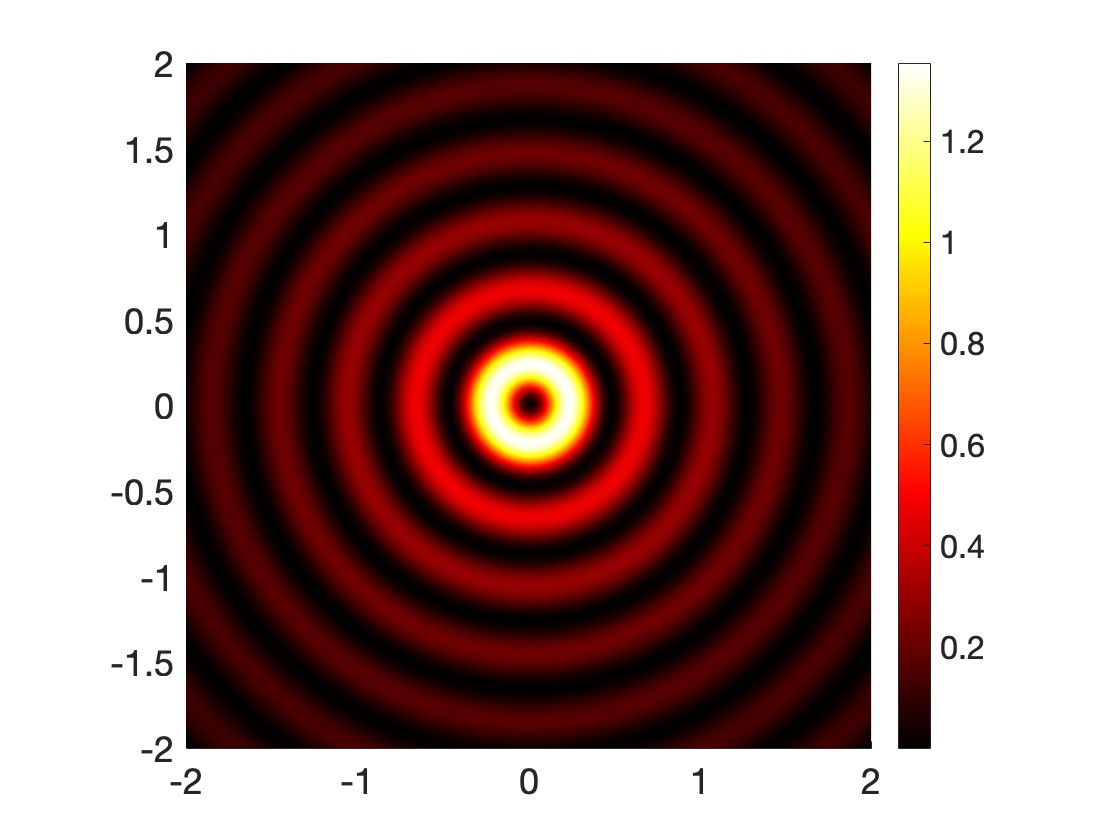}} \hspace{0cm}\\
\subfloat[$|\partial_{y_1}J_0(16|y-z|)|^2$]{\includegraphics[width=5.5cm]{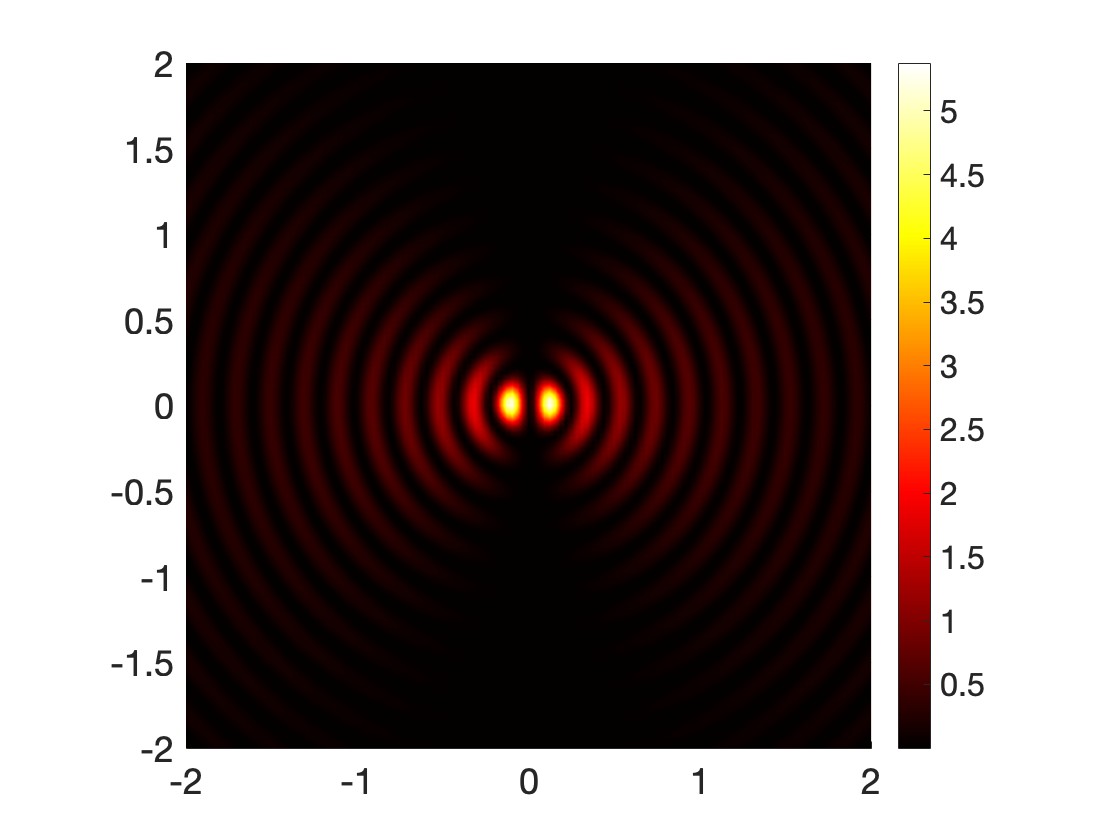}} \hspace{-0.7cm}
\subfloat[$|\partial_{y_2}J_0(16|y-z|)|^2$]{\includegraphics[width=5.5cm]{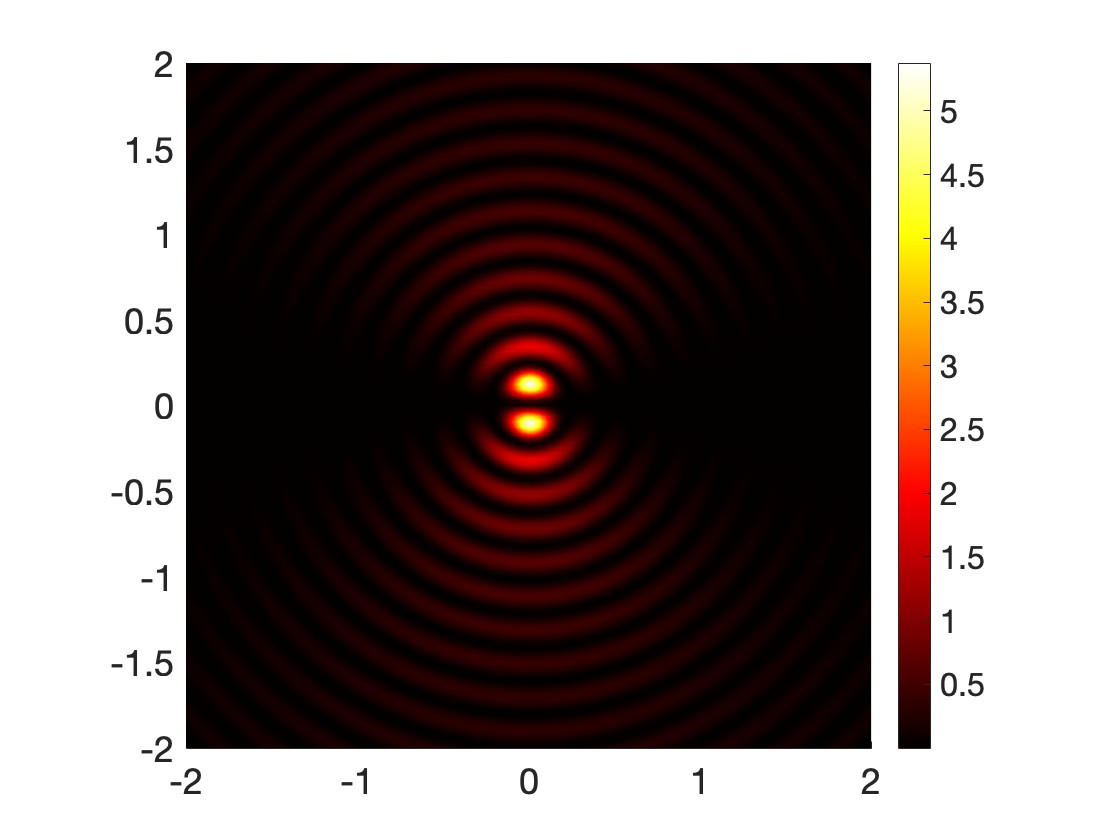}} \hspace{-0.7cm}
\subfloat[$|\nabla_yJ_0(16|y-z|)|^2$]{\includegraphics[width=5.5cm]{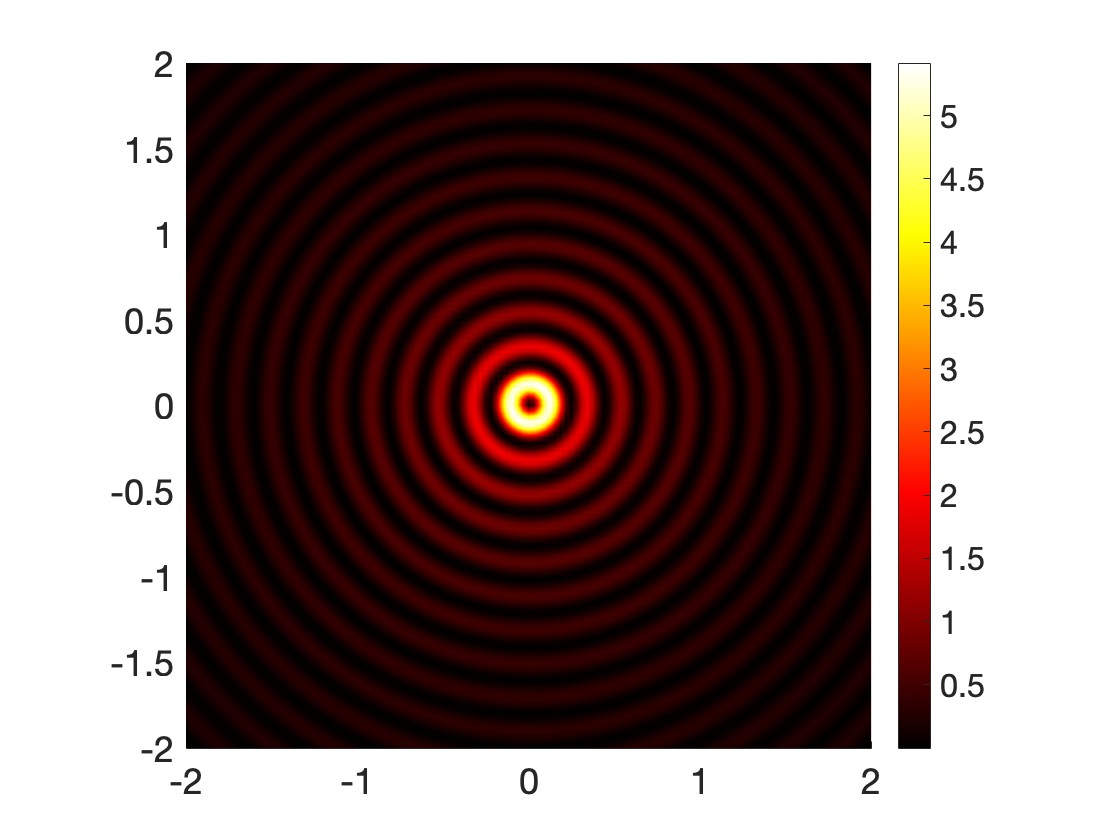}} \hspace{0cm}\\ 
\caption{Plots illustrating   each component and the modulus of  the 2D kernel  $\mathcal{K}(y,z)$ for $y = 0$, $z \in (-2,2)^2$, $k = 8$, and $k = 16$.
First row (a, b, c): plots of $|\partial_{y_1}J_0(k|y-z|)|^2, |\partial_{y_2}J_0(k|y-z|)|^2$ and $|\nabla_yJ_0(k|y-z|)|^2$ for $k = 8$. 
Second row (d, e, f): plots of $|\partial_{y_1}J_0(k|y-z|)|^2, |\partial_{y_2}J_0(k|y-z|)|^2$ and $|\nabla_yJ_0(k|y-z|)|^2$ for $k = 16$.} 
 \label{fi0}
\end{figure}

\begin{proof}
For the sake of simplicity in presentation, we omit the dependence of the total wave $u$ on $d_j$. From the integro-differential   equation~\eqref{LS} and  through a
 direct calculation  we derive that 
\begin{align}
& \int_{\partial \Omega} \left(\frac{\partial \Im \Phi(x,z)}{\partial \nu(x)} u_\sc(x) - \Im \Phi(x,z) \frac{\partial u_\sc(x)}{\partial \nu(x)} \right) \dd s(x) \nonumber \\
 & =  \int_{\partial\Omega} \left (\frac{\partial \Im \Phi(x,z)}{\partial \nu(x)} \div_x \int_D \Phi(x,y) Q(y) \nabla u(y) \right. \dd y  \nonumber \\
 & - \Im \Phi(x,z) \left.  \frac{\partial }{\partial \nu(x)} \div_x \int_D \Phi(x,y) Q(y) \nabla u(y)\dd y \right) \dd s(x)  \nonumber \\
  & =  \int_{\partial\Omega} \left (\frac{\partial \Im \Phi(x,z)}{\partial \nu(x)}  \int_D \nabla_x  \Phi(x,y)  \cdot Q(y) \nabla u(y) \right. \dd y  \nonumber \\
 & - \Im \Phi(x,z) \left.  \frac{\partial }{\partial \nu(x)}  \int_D \nabla_x  \Phi(x,y) \cdot Q(y) \nabla u(y)\dd y \right) \dd s(x).  \nonumber 
\end{align}
Now exchanging the integrals  and employing the fact that $\nabla_x \Phi(x,y) = - \nabla_y \Phi(x,y)$ we derive
\begin{align}
\label{int1}
&\int_{\partial \Omega} \left(\frac{\partial \Im \Phi(x,z)}{\partial \nu(x)} u_\sc(x) - \Im \Phi(x,z) \frac{\partial u_\sc(x)}{\partial \nu(x)} \right) \dd s(x) \nonumber \\
 &= - \int_D \nabla_y  \int_{\partial \Omega}  \left(\frac{\partial \Im \Phi(x,z)}{\partial \nu(x)} \Phi(y,x) - \Im \Phi(x,z) \frac{\partial \Phi(y,x)}{\partial \nu(x)} \right)\dd s(x)  \cdot Q(y) \nabla u(y) \dd y.
\end{align}
It is known (see~\cite{Colto2013}) that for  all $z, y \in \mathbb{R}^n$
$$\Delta \Im \Phi(z,y) + k^2 \Im \Phi(z,y) = 0.$$ 
Hence,  we deduce  from the Green  representation formula that
$$
  \int_{\partial \Omega}  \left(\frac{\partial \Im \Phi(x,z)}{\partial \nu(x)} \Phi(y,x) - \Im \Phi(x,z) \frac{\partial \Phi(y,x)}{\partial \nu(x)} \right)\dd s(x) = \Im \Phi(y,z).
$$
Therefore, substituting this identity in~\eqref{int1} implies
$$
 \int_{\partial \Omega} \left(\frac{\partial \Im \Phi(x,z)}{\partial \nu(x)} u_\sc(x) - \Im \Phi(x,z) \frac{\partial u_\sc(x)}{\partial \nu(x)} \right) \dd s(x)  = - \int_D \nabla_y \Im \Phi(y,z) \cdot Q(y) \nabla u(y) \dd y.
 $$
 Replacing this equation in~\eqref{I1} and letting $ \mathcal{K}(y,z) = \nabla_y \Im \Phi(y,z)$ we derive~\eqref{thm1}. 
 
 Now we justify~\eqref{decayrate} for the two-dimensional case.
 First we have
 \begin{align*}
 \nabla_y J_0(k|y-z|) = - k\left(\frac{(y_1 - z_1)}{|y-z|}, \frac{(y_2 - z_2)}{|y-z|}\right)^\top J_1(k|y-z|),
 \end{align*}
 where $J_1$ is a Bessel function of the first kind and first order. Using this and 
the Cauchy-Schwartz inequality we obtain
 $$
 \left|\int_{D}   \nabla_y J_0(k|y-z|) \cdot Q(y) \nabla u(y,d_j)\, \dd y \right|^p \leq \|k J_1(k|\cdot-z|) \|^p \| Q \nabla u(\cdot,d_j)\|^p
 $$
  Now~\eqref{decayrate} will follow from the well-known asymptotic behavior of $J_1$
 $$
   J_1(k|y-z|) = O\left(\frac{1}{\sqrt{|y-z|}}\right), \quad  \text{ as } |y-z| \to \infty. 
 $$
 The three-dimensional case can be done similarly.
\end{proof}
In   practice  the data are always perturbed with some noise. We assume the noisy data $u^\delta_{\sc}$ and $\partial u^\delta_{\sc}/\partial \nu$ satisfy    
\begin{align}
\label{noise1}
\sum_{j = 1}^N\| u_\sc(\cdot,d_j) - u^\delta_{\sc}(\cdot,d_j) \|_{L^2(\partial \Omega)}  \leq \delta_1,  \\
\label{noise2}
\sum_{j = 1}^N \left\| \frac{\partial u_\sc(\cdot,d_j) }{\partial \nu} - \frac{\partial u^\delta_{\sc}(\cdot,d_j) }{\partial \nu} \right \|_{L^2(\partial \Omega)}  \leq \delta_2,
\end{align}
for some positive constants $\delta_1$ and  $\delta_2 $. We now prove a stability estimate for the imaging  functional  $I(z)$.

\begin{theorem}
\label{stability} 
Denote   by $I^{\delta}(z)$ the imaging functional corresponding to noisy data $ u^\delta_{\sc}$ and $\partial u^\delta_{\sc}/\partial \nu$, i.e. 
$$
I^\delta(z) := \sum_{j  = 1}^N \left| \int_{\partial \Omega} u^\delta_\sc(x,d_j)  \frac{\partial\Im \Phi(x,z)}{\partial \nu(x)} - \frac{\partial u^\delta_\sc(x,d_j)}{\partial \nu(x)}  \Im \Phi(x,z)\dd s(x)  \right|^p.
$$
 Then for all $z \in \R^3$,
\begin{align*}
|I(z) - I^{\delta}(z)| = O(\max\{\delta_1,\delta_2\}) \quad \text{as}\ \max\{\delta_1,\delta_2\}\to 0.
\end{align*}
\end{theorem}
\begin{proof} 
For $j  = 1, 2, \dots, N$, we denote by
\begin{align*}
\alpha_j = \left| \int_{\partial \Omega} u_\sc(x,d_j)  \frac{\partial\Im \Phi(x,z)}{\partial \nu(x)} - \frac{\partial u_\sc(x,d_j)}{\partial \nu(x)}  \Im \Phi(x,z)\dd s(x)  \right|, \\
\beta_j = \left| \int_{\partial \Omega} u^\delta_\sc(x,d_j)  \frac{\partial\Im \Phi(x,z)}{\partial \nu(x)} - \frac{\partial u^\delta_\sc(x,d_j)}{\partial \nu(x)}  \Im \Phi(x,z)\dd s(x)  \right|. 
\end{align*}
Then using the triangle and Cauchy-Schwarz inequalities and~\eqref{noise1}-\eqref{noise2} we estimate that
\begin{align*}
|\alpha_j - \beta_j|
&\leq \left\| \frac{\partial\Im \Phi(\cdot,z)}{\partial \nu}  \right\| \| u_\sc(\cdot,d_j) -  u^\delta_\sc(\cdot,d_j) \| + \|\Im \Phi(\cdot,z)\| \left\|\frac{\partial u_\sc(\cdot,d_j)}{\partial \nu} - \frac{\partial u^\delta_\sc(\cdot,d_j)}{\partial \nu} \right\| \\
&\leq \sqrt{\left\| \frac{\partial\Im \Phi(\cdot,z)}{\partial \nu}  \right\|^2 +  \|\Im \Phi(\cdot,z)\|^2 } \max\{\delta_1,\delta_2\}.
\end{align*}
For the convenience of the presentation we always use $C$ as a generic constant. With the help of the estimate above for $|\alpha_j - \beta_j|$ and the triangle inequality we 
have that
\begin{align*}
|I^\delta(z) - I(z)| &\leq \sum_{j=1}^N \left|\beta_j^p -\alpha_j^p\right| 
\leq \left(\sum_{j=1}^N\left|\beta_j - \alpha_j\right| \right) \left(\sum_{l=1}^N \left| \sum_{m=0}^{p-1} \beta_j^m \alpha_j^{p-1-m}\right|\right) \\
&\leq C \max\{\delta_1,\delta_2\} \sum_{j=1}^N \sum_{m=0}^{p-1} (\left|\beta_j - \alpha_j\right|+\alpha_j)^m \alpha_j^{p-1-m}.
\end{align*}
Then denoting $\alpha_{\max} = \max\limits_{j = 1,\dots,N} \alpha_j$ and using the estimate $(a+b)^m \leq 2^m(a^m+b^m)$ for nonnegative numbers $a, b$ and $m = 0, \dots, p-1$ 
we continue to estimate that
\begin{align*}
|I^\delta(z) - I(z)| &\leq C \max\{\delta_1,\delta_2\} \sum_{j=1}^N \sum_{m=0}^{p-1}2^m (\left|\beta_j - \alpha_j \right|^m+\alpha_j^m)\alpha_j^{p-1-m} \\
&\leq C \max\{\delta_1,\delta_2\} \sum_{m=0}^{p-1} 2^m \alpha_{\max}^{p-1-m} \sum_{j=1}^N (\left| \beta_j - \alpha_j \right|^m+\alpha_{\max}^m) \\
&\leq C \max\{\delta_1,\delta_2\} \sum_{m=0}^{p-1}2^m \left(C^m\max\{\delta_1,\delta_2\}^m \alpha_{\max}^{p-1-m}+N\alpha_{\max}^{p-1}\right)\\
& = O(\max\{\delta_1,\delta_2\}),\quad \text{as}\ \max\{\delta_1,\delta_2\}\to 0.
\end{align*}
This completes the proof.
\end{proof}

\begin{remark}
It is important to observe that if the far-field measurements are acquired along the boundary of a large-radius ball, 
utilizing the radiation condition allows us to  approximate $\partial u_\sc/\partial \nu$ by $ik u_\sc$ 
within  $I(z)$. Then the modified imaging functional
\begin{equation}
\label{Ifar}
I_{\mathrm{far}}(z) := \sum_{j  = 1}^N \left| \int_{\partial \Omega} u_\sc(x,d_j)  \frac{\partial\Im \Phi(x,z)}{\partial \nu(x)} - ik u_\sc(x,d_j)  \Im \Phi(x,z)\dd s(x)  \right|^p 
\end{equation}
only needs  
the scattered field data $u_\sc(x,d)$ and approximates $I(z)$.
\end{remark}

\subsection{Numerical Examples}
\label{sec5}
In this section we study the numerical performance of the imaging functional~\eqref{I1}  for  simulated data in two dimensions. More precisely, we will examine the performance of the functional  for data with different wave numbers (Figure~\ref{fi1}), different number of incident waves (Figure~\ref{fi2}), and different powers $p$ (Figure~\ref{fi3}).
To ensure a consistent comparison, the imaging functional is normalized by dividing it with its maximal value. Its value is thus between 0 and 1.
  
  The following common parameters  are used in the numerical examples 
\begin{align*}
 &\partial \Omega = \{ (x_1,x_2)^\top \in \R^2: x_1^2 + x_2^2 = 3^2 \}, \\
  & \text{Number of data points on } \partial \Omega \text{ : } 64  \ (\text{uniformly distributed}), \\
&\text{Sampling  domain}: (-2,2)\times(-2,2), \\
&\text{Number of sampling points}: 96\times 96.
\end{align*}
  The following scattering objects are considered in the numerical examples.
  
\noindent
a) Kite-ellipse  object 
\begin{align*}
\mathrm{kite} &= \{x \in  \R^2: x=((\cos(t) + 0.65\cos(2t) - 0.65)/2,  1.5\sin(t)/2.5)^\top,  0\leq t \leq 2\pi\}, \\
\mathrm{ellip} & = \{(x_1,x_2)^\top \in \R^2: (x-1)^2/0.4^2 + (x_2 +0.75)^2/0.6^2 < 1  \} \\
D &= \mathrm{disk} \cup \mathrm{rectangle},  \\ 
Q(x)  &= \mathrm{diag}(0.5,0.7) \quad \text{in } D.
\end{align*}
b) Rectangle  object 
\begin{align*}
D &= \{(x_1,x_2)^\top \in \R^2: |x_1-0.2| < 0.6, |x_2-0.2| < 0.6\}, \\
Q(x)  &= \mathrm{diag}(0.5,0.7) \quad \text{in } D.
\end{align*}
Here for simplicity we consider $Q$ as a diagonal matrix with constant entries in $D$. It's worth noting that the well-posedness of the direct problem (see~\cite{Kirsc2007b}) requires that  $Q+I_n$ is  symmetric and positive definite. 
Using $N$ incident plane waves and measuring the data at $64$ points on $\partial \Omega$ for each incident wave, 
the Cauchy data $u_\sc(x,d), \partial u_\sc/ \partial \nu(x,d)$, where $(x,d) \in \partial \Omega \times \S$, are then $N \times 64$ matrices. 
The artificial noise is simulated as 
two complex-valued noise matrices $\mathcal{N}_{1,2}$ 
containing random numbers that are uniformly distributed in the complex square 
$$\{ a+ i b\, : \, |a| \leq 1, \, |b| \leq 1 \} \subset \C.$$ For simplicity we consider the same noise level  $\delta$ for both $u_\sc$ and $\partial u_\sc/ \partial \nu$.  
The  noisy data $u_\sc^\delta$ and $\partial u^\delta_\sc/ \partial \nu$ are given by
\begin{align*}
u_\sc^\delta 
 := u_\sc 
   + \delta\frac{\mathcal{N}_1}{\|\mathcal{N}_1\|_F} \|u_\sc  \|_F, \quad 
\frac{\partial u_\sc^\delta }{\partial \nu} := \frac{\partial u_\sc }{\partial \nu}  
   + \delta\frac{\mathcal{N}_2}{\|\mathcal{N}_2\|_F} \left\|\frac{\partial u_\sc }{\partial \nu} \right \|_F,
\end{align*}
where $\|\cdot\|_F$ is the Frobenius matrix norm.

\subsubsection{Reconstruction with different wave numbers (Figure \ref{fi1})}
We present the reconstruction results in Figure~\ref{fi1}   for the wave numbers $k = 8$ (wavelength $\approx$ 0.78) 
and $k = 16$ (wavelength $\approx$ 0.39).  For the data, we use 32 incident waves and 
 introduce  2$\%$ noise ($\delta = 0.02$). This example intentionally employs a small amount of noise in the data to distinctly display the impact of using different wave numbers in the reconstruction process. Figure~\ref{fi1} illustrates that the reconstruction results notably improve with higher values of $k$. 

\begin{figure}[h!!!]
\centering
\subfloat[True geometry]{\includegraphics[width=5.5cm]{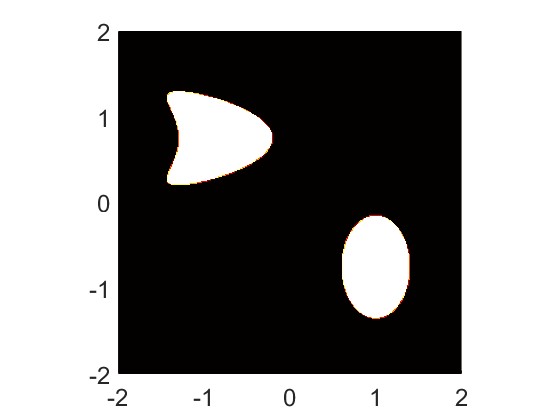}} \hspace{-0.6cm}
\subfloat[$k = 8$]{\includegraphics[width=5.5cm]{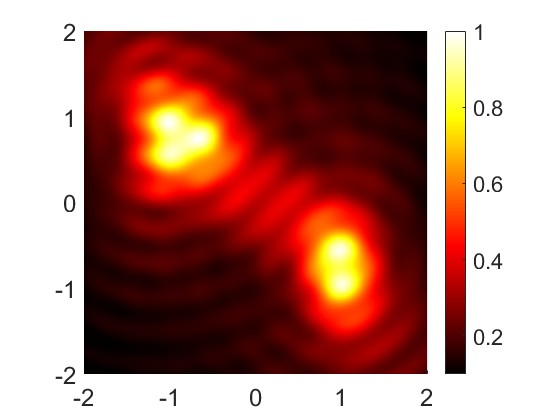}} \hspace{-0.6cm}
\subfloat[$k= 16$]{\includegraphics[width=5.5cm]{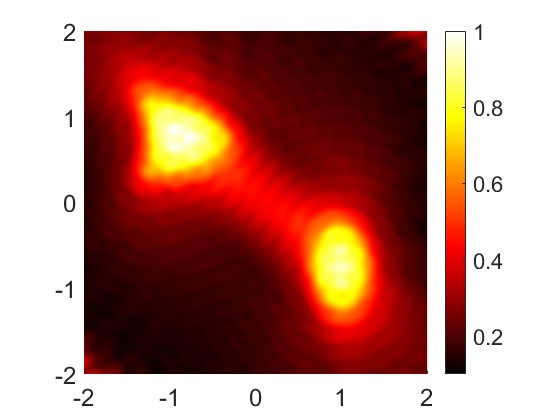}} \hspace{0cm}\\
\subfloat[True geometry]{\includegraphics[width=5.5cm]{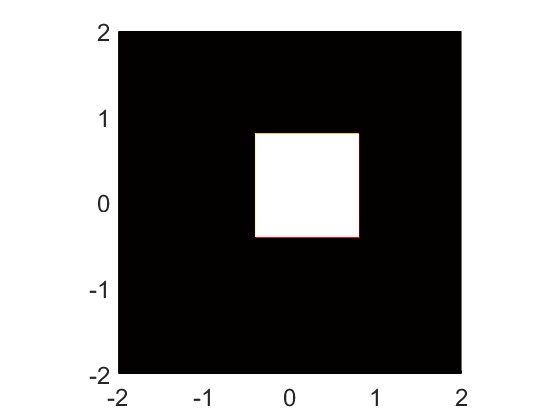}}  \hspace{-0.6cm}
\subfloat[$k = 8$]{\includegraphics[width=5.5cm]{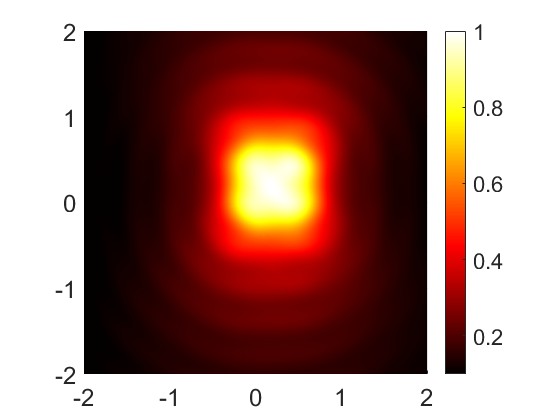}} \hspace{-0.6cm}
\subfloat[$k = 16$]{\includegraphics[width=5.5cm]{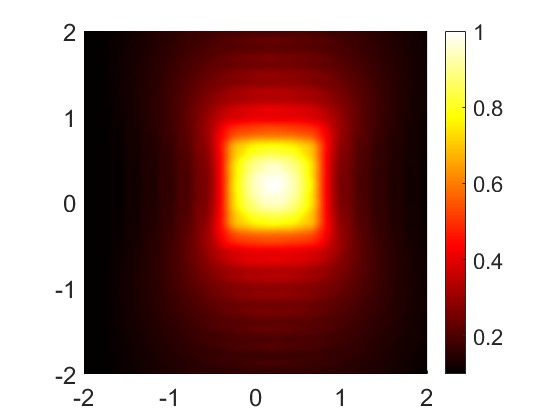}}  
\caption{Reconstruction from   data associated with different wave numbers.  
The data is generated using 32 incident waves and
has a  2$\%$ noise level.
First column (a, d): true geometry. Second column (b, e): reconstruction with $k = 8$.
Third column (c, f): reconstruction with $k = 16$.} 
 \label{fi1}
\end{figure}

\subsubsection{Reconstruction with different numbers of incident waves (Figure \ref{fi2})}
We present the reconstruction results in Figure~\ref{fi2}  for noisy data using 8 incident waves and 32 incident waves. The data is added with 
 20$\%$ noise and correspond to  the wave number $k  = 16$. It is clearly seen that the results significantly improve  with an increased count of incident waves used to generate the 
 data.  However, it is worth noting that the results will remain relatively stable even as more incident waves are used in the reconstruction process.

\begin{figure}[h!!!]
\centering
\subfloat[True geometry]{\includegraphics[width=5.5cm]{kite_ellip_true}} \hspace{-0.6cm}
\subfloat[$N = 8$]{\includegraphics[width=5.5cm]{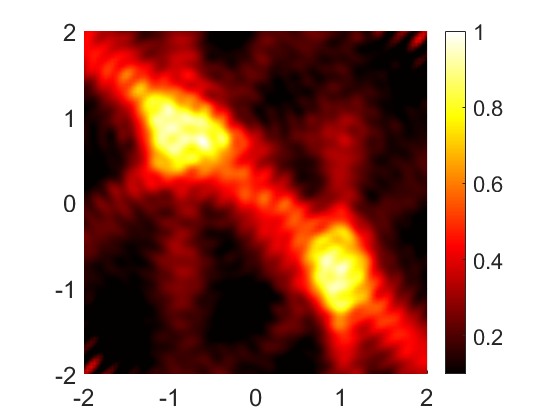}} \hspace{-0.6cm}
\subfloat[$N= 32$]{\includegraphics[width=5.5cm]{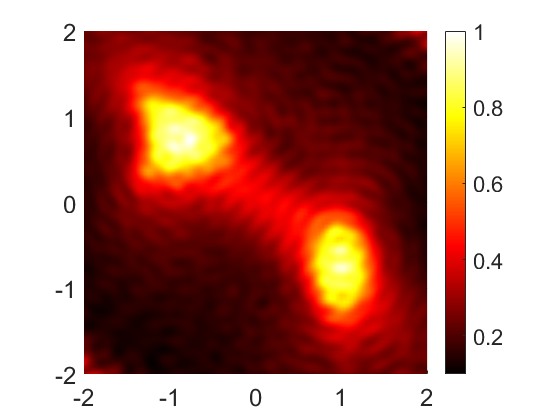}} \hspace{-0.6cm}\\
\subfloat[True geometry]{\includegraphics[width=5.5cm]{rec_true}}  \hspace{-0.6cm}
\subfloat[$N = 8$]{\includegraphics[width=5.5cm]{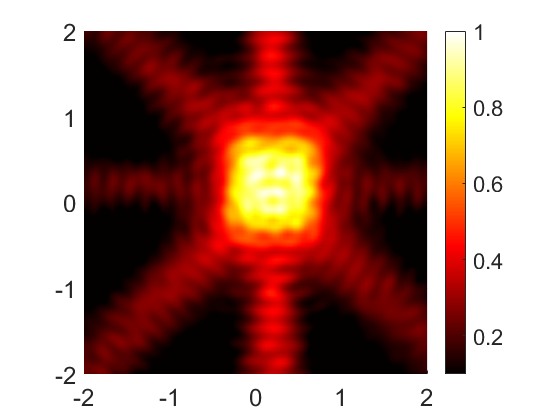}} \hspace{-0.6cm}
\subfloat[$N = 32$]{\includegraphics[width=5.5cm]{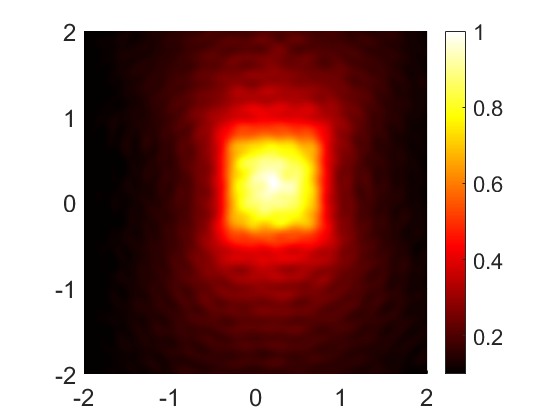}}  
\caption{Reconstruction with noisy  data using different numbers of incident waves.  
The wave number is $k = 16$ and the data is added by 20$\%$ noise.
First column (a, d): true geometry. Second column (b, e): reconstruction using 8 incident waves ($N = 8$).
Third column (c, f): reconstruction using 32 incident waves ($N = 32$).} 
 \label{fi2}
\end{figure}

\subsubsection{Reconstruction with different  powers $p$ (Figure \ref{fi3})}
Figure~\ref{fi3} presents the reconstruction results using different values of $p$ in $I(z)$. We consider a 20\% noise level in the data, a wave number of $k = 16$, and again, 32 incident waves. While the reconstructions with $p = 1$ display reasonable locations and shapes of the targets, they are not as sharp or visually appealing as those in the case of $p = 2$ shown in Figure~\ref{fi2}. The results with $p = 4$ are even less accurate in terms of the shape of the computed objects. These two experiments indicate that using $p = 2$ seems to yield the most reasonable results.

\begin{figure}[h!!!]
\centering
\subfloat[True geometry]{\includegraphics[width=5.5cm]{kite_ellip_true}} \hspace{-0.6cm}
\subfloat[ $p = 1$]{\includegraphics[width=5.5cm]{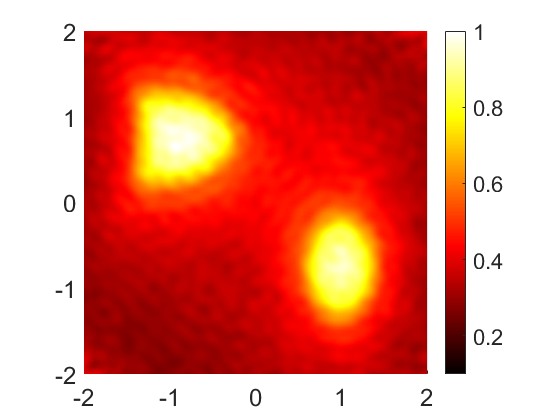}}  \hspace{-0.6cm} 
\subfloat[ $p = 4$]{\includegraphics[width=5.5cm]{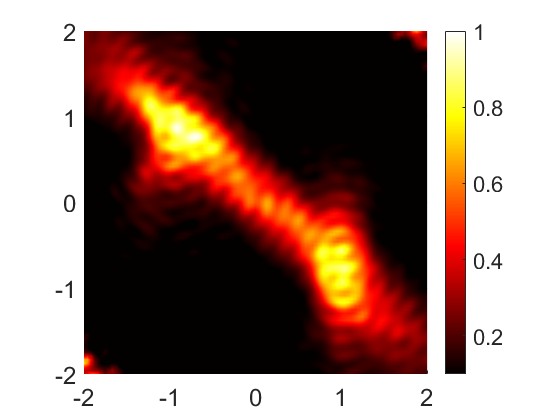}} \\
\subfloat[True geometry]{\includegraphics[width=5.5cm]{rec_true}} \hspace{-0.6cm}
\subfloat[$p = 1$]{\includegraphics[width=5.5cm]{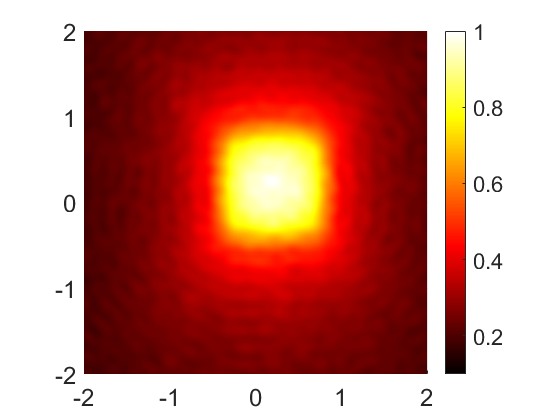}}  \hspace{-0.6cm} 
\subfloat[ $p = 4$]{\includegraphics[width=5.5cm]{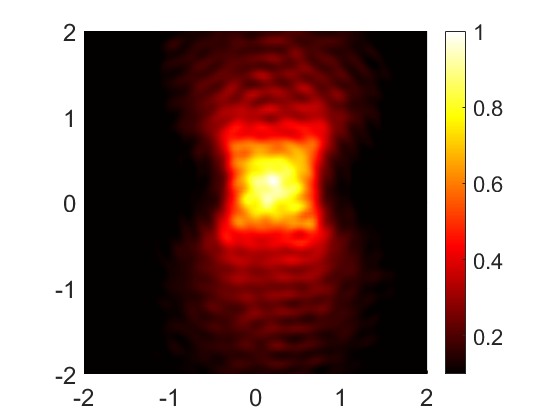}}  
\caption{Reconstruction  with  noisy data (20$\%$ noise) using different values of $p$, 32 incident waves, wave number $k = 16$.  
First column (a, d): true geometry. Second column (b, e): reconstruction with $p = 1$.
Third column (c, f): reconstruction with $p = 4$.} 
 \label{fi3}
\end{figure}

\section{The Case of Maxwell's Equations}

We consider the scattering of time-harmonic electromagnetic waves at  positive wave number
$k$ from a non-magnetic  inhomogeneous medium. Suppose that there is no free charge and current density. 
Then, the electric field $ \E$ and the magnetic field $\H$ satisfy the Maxwell's equations 
\begin{align}
\label{eq:Maxwell}
\curl \E - i k  \H = 0 \quad \text{ and } \quad \curl \H +  ik \eps \E = 0 \quad \text{in } \R^3,
\end{align}
where  $\eps$ is the electric permittivity 
of the medium (the magnetic permeability is assumed to be 1). The permittivity $\eps$ is assumed to be a bounded  matrix-valued function. 
Let $D$ be a bounded domain occupied by the non-magnetic inhomogeneous medium. The medium outside 
$D$ is assumed to be
 homogeneous which means  
 $\eps = I_3$ outside $D$. Eliminating magnetic field $\H$ from~\eqref{eq:Maxwell} we obtain  
\begin{equation}
  \label{eq:Order2Total}
  \curl \curl \E   - k^2 \eps  \E  = 0, \quad   \text{in } \R^3.
\end{equation}
Assume that the inhomogeneous  medium is illuminated by the incident electric and magnetic incident fields $ \E_\mathrm{in}$ and $\H_\mathrm{in}$, 
respectively, satisfying
\begin{align*}
 \curl \H_\mathrm{in} + i k \Ei = 0 \quad \text{ and} \quad \curl \Ei - i k \H_\mathrm{in} = 0, \quad \text{in } \R^3.
\end{align*}
Then  there arises the scattered electric field 
$\u$, defined by $\u:= \E - \Ei$. 
Since the incident field $\Ei$ satisfies the homogeneous Maxwell equation with wave number $k$ given by 
$$\curl\curl \Ei - k^2  \Ei= 0, \quad   \text{in } \R^3$$
 subtracting this equation from~\eqref{eq:Order2Total} we can conclude that 
the scattered field  satisfies
\begin{align}
\label{secondorder}
  \curl \curl \u - k^2 \eps  \u =  k^2P\Ei  \quad  \text{in } \R^3,
\end{align}
where the contrast $P$ is defined by
\begin{equation*}
 P := \eps - I_3.
\end{equation*}
Therefore, $P = 0$ in $\R^3 \setminus \overline{D}$. 
We complete the scattering problem by the  Silver-M\"{u}ller radiation condition for the scattered field $\u$ given by 
\begin{align}
 \label{radcond}
\frac{\x}{|\x|}  \times \curl \u  - ik \u = \mathcal{O}(|\x|^{-2}) \quad \text{as } |\x| \rightarrow \infty, 
\end{align}
which is assumed to hold uniformly with respect to $\x/|\x|$. For the study of the inverse problem discussed below we assume that
the problem~\eqref{secondorder}--\eqref{radcond} is  well-posed. The well-posedness
of this direct problem was studied in~\cite{Kirsc2007a}.

Let $\Omega$ be a regular domain such that $\overline{D} \subset \Omega$ and $\nu$ be the outward normal unit on $\partial \Omega$. Consider 
$N$ incident fields $\Ei$ as plane waves
$$
\Ei(\x,\di_j)= \q_j e^{ik \di_j \cdot \x}, \quad \di_j \in \mathbb{S}^2, \quad j = 1, \dots, N,
$$
where  $\q_j$ is the polarization vector in $\R^3$ that is orthogonal to $\di_j$. 
We consider the following inverse problem.

\noindent
\textbf{Inverse problem.}  Given $\u(\cdot,\di_j)$ and $\nu  \times \curl \u(\cdot,\di_j)$ on  $\partial \Omega$ for $\di_j  \in \mathbb{S}^2$ where $j = 1,\dots, N$,
determine $D$.

\begin{remark}
From Maxwell equations~\eqref{eq:Maxwell}, it is evident  that the data $\u$ and $\nu \times \curl  \u$ can be respectively obtained  through the measurement of the electric field $\E$ and the tangential components of the magnetic field, $\nu \times \H$, on $\partial \Omega$.
\end{remark}
Similar to the scalar case, an integro-differential equation formulation for the scattered field $\u$ is needed to investigate the imaging functional. To this end, we need
the Green's tensor of the scattering problem~\eqref{secondorder}--\eqref{radcond} that is given by
\begin{align}
\label{tensor}
\mathbb{G}(\x,\y) = \Phi(\x,\y)I_3 + \frac{1}{k^2} \nabla_\x \div_\x(\Phi(\x,\y)I_3), \quad  \x\neq \y,
\end{align}
where 
$$
\Phi(\x,\y) = \frac{e^{ik|\x-\y|}}{4\pi|\x-\y|}.
$$
It is  known from~\cite{Kirsc2007a} that  the  scattered field $\u$ (weak solution in $H_{\mathrm{loc}}(\curl, \R^3)$)  in problem~\eqref{secondorder}--\eqref{radcond} can be described by the following the integro-differential equation 
\begin{align}
\label{LS2}
\u(\x) = (k^2 + \nabla \div) \int_{D} \Phi(\x,\y)  P(\y) \E(\y) \d \y, \quad  \x \in\R^3.
\end{align}

\subsection{Stable Imaging Functional}
Let $\p_n$ be vectors in $\S^2$, $n = 1, 2, \dots, M$. We define the imaging functional 
\begin{align}
\label{I2}
\mathcal{I}(\z) := \sum_{n = 1}^M \sum_{j=1}^N &\left| \int_{\partial \Omega} \left[ \nu \times \curl(\Im \mathbb{G}(\x,\z)\p_n) \right]\cdot \u(\x,\di_j) 
 - \Im \mathbb{G}(\x,\z)\p_n \cdot \left[ \nu \times \curl  \u(\x,\di_j) \right] \d s(\x)\right|^p, 
\end{align}
where $ \z \in \R^3$ and $p \in \N$ is again used to sharpen the significant peaks of $\mathcal{I}(\z)$.
\begin{remark}
Based on  numerical observations of the kernel function $\mathcal{K}(\y,\z,\p_n)$ in~\eqref{kernel3d} (see Figure~\ref{fi00}) and the numerical study of the imaging functional (see Figure~\ref{fi5}), using  multiple vectors $\p_n$ tends to yield improved reconstruction results. Moreover, similar to the scalar case, when the measurements are acquired along the boundary of a large-radius ball, utilizing the radiation condition allows us to  approximate $\nu \times \curl  \u$ by $ik \u$ 
in~\eqref{I2}.
\end{remark}

Recall that $j_0$ is the spherical Bessel function of the first kind. We study the behavior of the imaging functional  $\mathcal{I}(\z)$ in the following theorem.
\begin{theorem}
The imaging functional satisfies 
\begin{align}
\label{thm2}
\mathcal{I}(\z) =  \sum_{n = 1}^M \sum_{j=1}^N  \left|  \int_{D} \mathcal{K}(\y,\z, {\p_n}) \cdot P(\y) \E(\y, \di_j) \d \y \right|^p,
\end{align}
where
\begin{align}
\label{kernel3d}
\mathcal{K}(\y,\z, {\p_n}) &= - \frac{k\p_n| \y - \z|^2-k(\p_n\cdot ( \y - \z))(\y - \z)}{4\pi|\y - \z|^2}j_0(k| \y - \z|) \nonumber \\
&+ \frac{3(\p_n\cdot ( \y - \z))( \y - \z)-\p_n| \y - \z|^2}{k 4\pi | \y - \z|^4}(\cos(k| \y - \z|)-j_0(k| \y - \z|)).
\end{align}
Furthermore 
\begin{align}
\label{decay2}
\mathcal{I}(\z) = O\left(\frac{1}{\mathrm{dist}(\z,D)^p} \right)\quad  \text{ as } \mathrm{dist}(\z,D) \to \infty.
\end{align}
\end{theorem}

\begin{remark}
Similar to the scalar case, the kernel $\mathcal{K}(\y,\z, {\p_n})$ exhibits the desired peak property in magnitude which  partially explain the behavior of $\mathcal{I}(\z)$. 
The peak of the kernel in this case is even better than that of the scalar case, as seen in~Figure~\ref{fi00}. Moreover, the summation of  $|\mathcal{K}(\y,\z, {\p_n})|^2$ over multiple 
vectors $\p_n$ results in a stronger and more rounded peak,  potentially enhancing the performance of the imaging functional. This observation is confirmed in the numerical study.

\end{remark}

\begin{figure}[h!!!]
\centering
\subfloat[$|\mathcal{K}(0,\z, \p_1)|^2$]{\includegraphics[width=5.5cm]{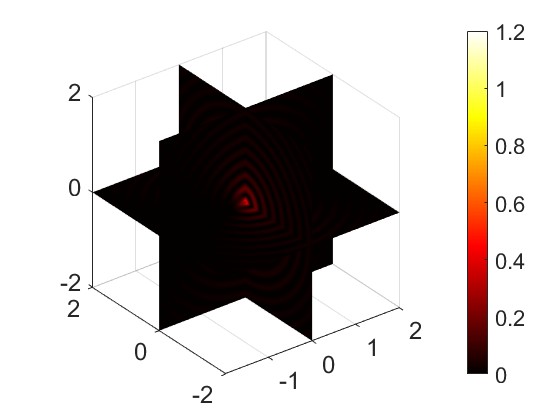}} \hspace{-0.4cm}
\subfloat[$|\mathcal{K}(0,\z, \p_2)|^2$]{\includegraphics[width=5.5cm]{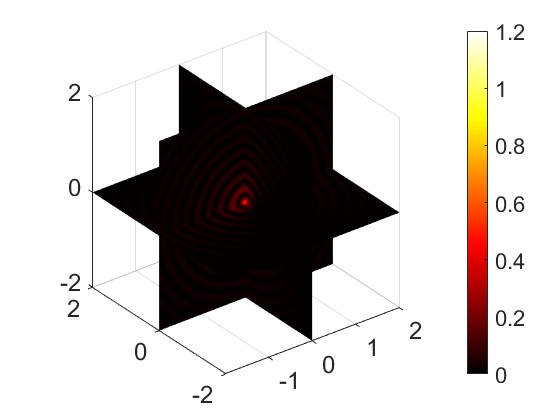}} \hspace{-0.4cm}
\subfloat[$\sum_{n=1}^3|\mathcal{K}(0,\z, \p_n)|^2$]{\includegraphics[width=5.5cm]{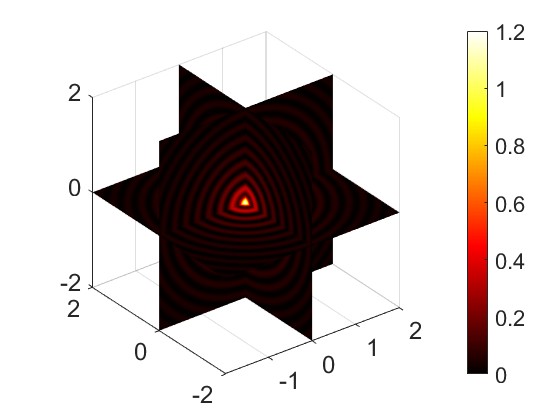}} \hspace{-0.4cm}
\subfloat[Top view of (a)]{\includegraphics[width=5.5cm]{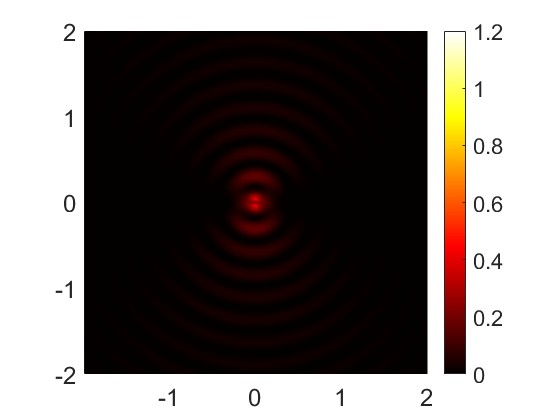}} \hspace{-0.6cm}
\subfloat[Top view of (b)]{\includegraphics[width=5.5cm]{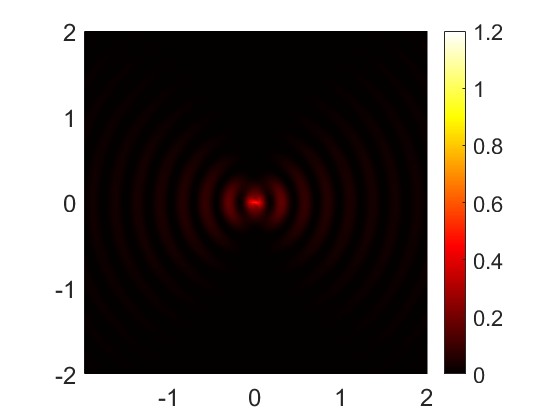}} \hspace{-0.6cm}
\subfloat[Top view of (c)]{\includegraphics[width=5.5cm]{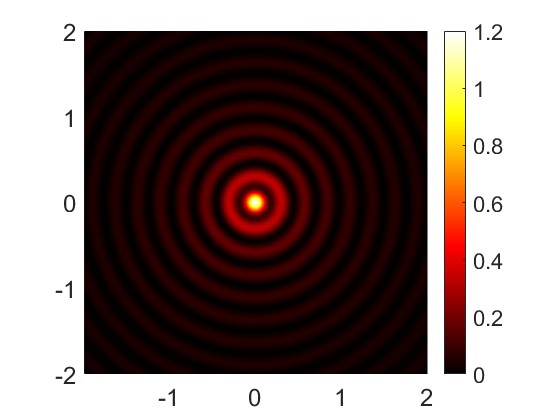}} \hspace{0cm}\\ 
\caption{Plots illustrating  $|\mathcal{K}(\y,\z, \p_1)|^2$, $|\mathcal{K}(\y,\z, \p_2)|^2$, and $\sum_{n=1}^3|\mathcal{K}(\y,\z, \p_n)|^2$, respectively, in (a), (b), and (c),  for $\y = 0$, $\z \in (-2,2)^2$,  $k = 12$,  and $\p_1 = (1,0,0)^\top, \p_2 = (0,1,0)^\top,  \p_3 = (0,0,1)^\top$.}
 \label{fi00}
\end{figure}

\begin{proof}
From~\eqref{LS2}  and a direct calculation we obtain 
\begin{align*}
&\u(\x) =  \int_{D} \mathbb{G}(\x,\y)  P(\y) \E(\y) \d \y, \\
&\nu \times \curl \u(\x) =  \int_{D} \nu \times   \curl_\x\left(\mathbb{G}(\x,\y) \right) P(\y) \E(\y) \d \y,
\end{align*}
where  $\curl$ of the Green's tensor is taken columnwise. 
A substitution leads to 
\begin{align}
& \int_{\partial \Omega} \left[ \nu \times \curl(\Im \mathbb{G}(\x,\z)\p_n) \right]\cdot \u(\x,\di_j) - \Im \mathbb{G}(\x,\z)\p_n \cdot \left[ \nu \times \curl  \u(\x,\di_j) \right] \d s(x) \nonumber\\
 & =  \int_{\partial \Omega} \left[ \nu \times \curl(\Im \mathbb{G}(\x,\z)\p_n) \right]\cdot  \int_{D} \mathbb{G}(\x,\y)  P(\y) \E(\y) \d \y \nonumber\\ 
 & - \Im \mathbb{G}(\x,\z)\p_n \cdot \left[    \int_{D} \nu \times \curl_\x \left[\mathbb{G}(\x,\y) \right] P(\y) \E(\y) \d \y \right] \d s(x)\nonumber \\
 & =  \int_{D} \int_{\partial \Omega}  \left[ \nu \times \curl(\Im \mathbb{G}(\x,\z)\p_n) \right]\cdot  \mathbb{G}(\x,\y)  P(\y) \E(\y)  \nonumber \\ 
 \label{id1}
 & - \Im \mathbb{G}(\x,\z)\p_n \cdot \left[     \nu \times \curl_\x \left(\mathbb{G}(\x,\y) \right) P(\y) \E(\y)  \right] \d s(x) \d \y.
\end{align}
Now we will compute the integral over $\partial \Omega$ in the  equation~\eqref{id1}. 
First from~\eqref{tensor} and a direct calculation we have that
$$
\Im\mathbb{G}(\x,\z)\p_n = \Im\Phi(\x,\z)\p_n + \frac{1}{k^2} \nabla_\x \div_\x(\Im\Phi(\x,\z)\p_n).
$$
Using this result, the identities  $\curl \nabla = 0$, $ \curl \curl = -\Delta + \nabla \div$, and the fact that 
\begin{align*}
\quad \Delta_\x\Im\Phi(\x,\z) = -k^2 \Im\Phi(\x,\z), \quad \text{for all } \x,\z \in \R^3,
\end{align*}
we derive
\begin{align}
\label{eq2}
\curl_\x \curl_\x \left(\Im \mathbb{G}(\x,\z)\p_n \right) = k^2 \Im \mathbb{G}(\x,\z)\p_n\quad \text{for all } \x,\z \in \R^3. 
\end{align}
In addition, for $ \q \in \R^3$, we have
\begin{align}
\label{eq1}
\curl_\x\curl_\x \left(\mathbb{G}(\x,\y)\q\right) - k^2 \mathbb{G}(\x,\y)\q = - \delta(\x-\y)\q.
\end{align}
Multiplying equations~\eqref{eq1} and~\eqref{eq2} by $\Im \mathbb{G}(\x,\z)\p_n$ and $\mathbb{G}(\x,\y)\q$, respectively, and integrating over $\Omega$ we obtain 
\begin{align*}
&\int_\Omega \curl_\x\curl_\x \left(\mathbb{G}(\x,\y)\q\right)\cdot  \Im \mathbb{G}(\x,\z)\p_n - k^2 \mathbb{G}(\x,\y)\q \cdot  \Im \mathbb{G}(\x,\z)\p_n \, \d \x = -\q \cdot  \Im \mathbb{G}(\y,\z)\p_n,\\
&\int_\Omega \curl_\x\curl_\x \left(\Im \mathbb{G}(\x,\z)\p_n \right)\cdot \mathbb{G}(\x,\y)\q - k^2 \Im \mathbb{G}(\x,\z)\p_n \cdot \mathbb{G}(\x,\y)\q \, \d \x = 0.
\end{align*}
Making a subtraction of the two equations above and integrating by parts we derive
\begin{align*}
&\int_{\partial \Omega}  \Im\mathbb{G}(\x,\z)\p_n\cdot \nu \times \curl_\x \mathbb{G}(\x,\y)\q - \nu \times \curl_\x \Im \mathbb{G}(\x,\z)\p_n \cdot \mathbb{G}(\x,\y)\q \,\d s(\x) \\
&= -\q \cdot  \Im \mathbb{G}(\y,\z)\p_n.
\end{align*}
Choosing $\q = P(\y) \E(\y)$ leads to
\begin{align*}
&\int_{\partial \Omega}  \Im\mathbb{G}(\x,\z)\p_n\cdot \nu \times \curl_\x \mathbb{G}(\x,\y)P(\y) \E(\y) 
- \nu \times \curl_\x \Im \mathbb{G}(\x,\z)\p_n \cdot \mathbb{G}(\x,\y)P(\y) \E(\y)\, \d s(\x) \\
&= -P(\y) \E(\y) \cdot  \Im \mathbb{G}(\y,\z)\p_n.
\end{align*}
Now substituting this result in~\eqref{id1} and comparing with $\mathcal{I}(\z)$ given in~\eqref{I2} we obtain 
$$
\mathcal{I}(\z) =  \sum_{n = 1}^M \sum_{j=1}^N \left| \int_{D}  \Im \mathbb{G}(\y,\z)\p_n \cdot P(\y) \E(\y) \d \y\right|^p.
$$
Setting $\mathcal{K}(\y,\z,\p_n) =  \Im \mathbb{G}(\y,\z)\p_n$ then 
$$
\mathcal{K}(\y,\z,\p_n) =  \frac{k}{4\pi} \left( j_0(k|\y-\z|)\p_n +  \frac{1}{k^2} \nabla_\y \div_\y( j_0(k|\y-\z|) \p_n)\right).
$$
Recall that $ j_0(k|\y-\z|) = \sin(k|\y-\z|)/(k|\y-\z|)$, a straightforward calculation implies
\begin{align*}
\mathcal{K}(\y,\z,\p_n) &= - \frac{k\p_n| \y - \z|^2-k(\p_n\cdot ( \y - \z))(\y - \z)}{4\pi|\y - \z|^2} j_0(k| \y - \z|) \\
&+ \frac{3(\p_n \cdot ( \y - \z))( \y - \z)-\p_n| \y - \z|^2}{k 4\pi | \y - \z|^4}(\cos(k| \y - \z|)-j_0(k| \y - \z|)).
\end{align*}
Using the Cauchy-Schwartz and triangle inequalities we obtain 
$$
\left|\mathcal{K}(\y,\z,\p_n)\right| \leq \frac{k}{2\pi} \left| j_0(k| \y - \z|)\right| + \frac{1 + \left|j_0(k| \y - \z|) \right|}{k\pi|\y-\z|}.
$$
From this estimate and the fact that $j_0(k|\y-\z|)  = O(1/|\y-\z|)$ as $|\y - \z| \to \infty$ we imply that
$$
\left|\mathcal{K}(\y,\z,\p_n)\right|  = O\left(\frac{1}{| \y - \z|}\right), \quad \text{as }  |\y - \z| \to \infty,
$$
which can help us deduce~\eqref{decay2} similarly as in the scalar case.

\end{proof}

We assume the noisy data $\u^\delta$ and $\nu \times \curl \u^\delta$ satisfy    
\begin{align*}
&\sum_{j = 1}^N\| \u(\cdot,\di_j) - \u^\delta(\cdot,\di_j) \|_{(L^2(\partial \Omega))^3}  \leq \delta_1,  \\
&\sum_{j = 1}^N \left\| \nu \times \curl \u(\cdot, \di_j)  - \nu \times \curl \u^\delta(\cdot, \di_j)  \right \|_{(L^2(\partial \Omega))^3}  \leq \delta_2,
\end{align*}
for some positive constants $\delta_1, \delta_2 $. We now have a stability estimate for the imaging  functional  $\mathcal{I}(\z)$.

\begin{theorem}
\label{stability2} 
Denote   by $\mathcal{I}^{\delta}(\z)$ the  imaging  functional corresponding to noisy data $ \u^\delta$ and $\nu \times \curl \u^\delta$, that means, 
\begin{align*}
\mathcal{I}^\delta(\z) = 
\sum_{n = 1}^M \sum_{j=1}^N \left| \int_{\partial \Omega} \left[ \nu \times \curl(\Im \mathbb{G}(\x,\z)\p_n) \right]\cdot \u^\delta(\x,\di_j) - \Im \mathbb{G}(\x,\z)\p_n \cdot \left[ \nu \times \curl  \u^\delta(\x,\di_j) \right] \d s(x)\right|^p. 
\end{align*}
 Then for all $\z \in \R^3$,
\begin{align*}
|\mathcal{I}(\z) - \mathcal{I}^\delta(\z)  | = O(\max\{\delta_1,\delta_2\}) \quad \text{as}\ \max\{\delta_1,\delta_2\}\to 0.
\end{align*}
\end{theorem}
\begin{proof}
The proof can be done similarly as in the scalar case and is therefore omitted.
\end{proof}

\subsection{Numerical Examples}
We provide numerical examples to effectively illustrate the efficiency of the imaging functional $\mathcal{I}(\z)$ in reconstructing  scatterers using simulated data. Specifically,  we  examine the performance of the functional  using different numbers of incident waves (Figure~\ref{fi4}), using different sets of  vectors $\p_n$ for $n = 1, \dots, M$ (Figure~\ref{fi5}),  and  using highly noisy data (Figure~\ref{fi6}).
As in the scalar case, the imaging functional is normalized by dividing it with its maximal value.
  
  The following common parameters  are used in the numerical examples 
\begin{align*}
 &\partial \Omega = \{ (x_1,x_2,x_3)^\top \in \R^2: x_1^2 + x_2^2 + x_3^2 = 3^2 \}, \\
  & \text{Number of data points on } \partial \Omega \text{ : } 324  \ (\text{uniformly distributed}), \\
&\text{Sampling  domain}: (-1.5,1.5)^3, \\
&\text{Number of sampling points}: 60^3,\\
& \text{Power } p \text{ in } I(z): p = 2,\\
& \text{Isovalue for  3D plots in Matlab : } 0.4.
\end{align*}
The choice of $p = 2$ follows from the numerical study in the 2D case. 
  The following  scattering object consisting a sphere ($D_1$) and an L shape ($D_2\setminus D_3$) is considered in the numerical examples in Figures~\ref{fi4}--\ref{fi6}
\begin{align*}
D_1 &= \left\{\x \in \R^3: |\x - \mathbf{a}|^2 < 0.5^2, \mathbf{a} = (0.6,0,0.5)^\top \right \}, \\
D_2 &= \left\{\x \in \R^3: -1.1 < x_1 < 0.1, -0.3< x_2 < 0.3, -1< x_3 < 0.2  \right \}, \\
D_3 &= \left\{\x \in \R^3: -0.5 < x_1 < 0.1, -0.3< x_2< 0.3, -0.4< x_3< 0.2   \right \}, \\
Q(\x) &= \begin{cases} 
\mathrm{diag}(0.5,0.4,0.3), \quad & \x \in D_1 \cup (D_2\setminus \overline{D_3}) \\
  0, \quad &  \text{else}.
    \end{cases}
\end{align*}
We consider the incident plane waves 
$$
\Ei(\x,\di_j)= (\di_j \times (1, 0, 0)^\top) e^{ik \di_j \cdot \x}, \quad \di_j \in \mathbb{S}^2, \quad j = 1, \dots, N,
$$
where the wave vectors $\di_j$ were generated
almost uniformly on $\mathbb{S}^2$.
It was observed during the numerical study that the choice of using $(1,0,0)^\top$ in the polarization of $\Ei(\x,\di_j)$ or other vectors did not appear to significantly impact the reconstruction results.
We simulated artificial noise using complex-valued noise matrices, similar to the approach used in the Helmholtz equation case.

\subsubsection{Reconstruction with different wave numbers  (Figure \ref{fi6})}

Figure~\ref{fi6} presents  the reconstruction results    for the wave numbers $k = 6$ (wavelength $\approx$ 1.04) 
and $k = 12$ (wavelength $\approx$ 0.52). We consider 180 incident waves to generate the data, which is   perturbed by a  2$\%$ noise level ($\delta = 0.02$).
As in the scalar case, we intentionally introduce a small amount of noise into the data to distinctly demonstrate the impact of using different wave numbers in the reconstruction process.
 Figure~\ref{fi6} illustrates that the reconstruction results notably improve with higher values of $k$.

\begin{figure}[h!!!]
\centering
\subfloat[True geometry]{\includegraphics[width=6.3cm]{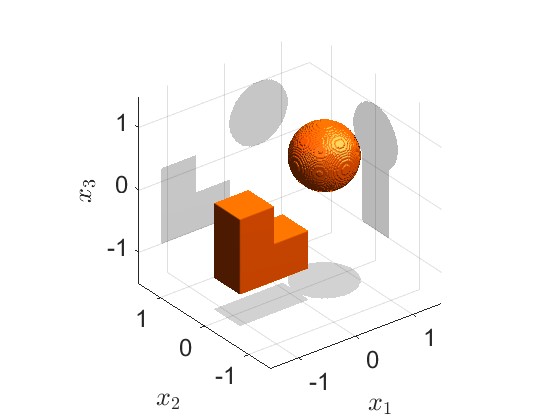}} \hspace{-1.45cm}
\subfloat[$k = 6$]{\includegraphics[width=6.3cm]{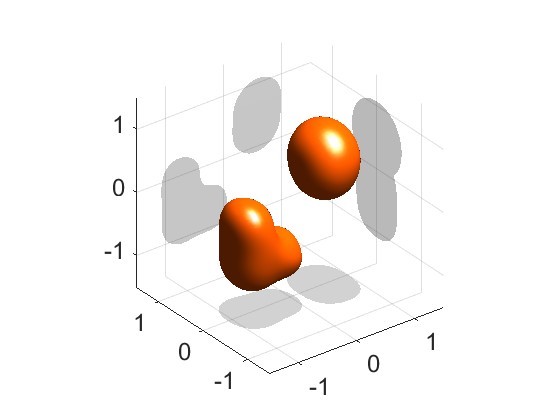}} \hspace{-1.45cm}
\subfloat[$k = 12$ ]{\includegraphics[width=6.3cm]{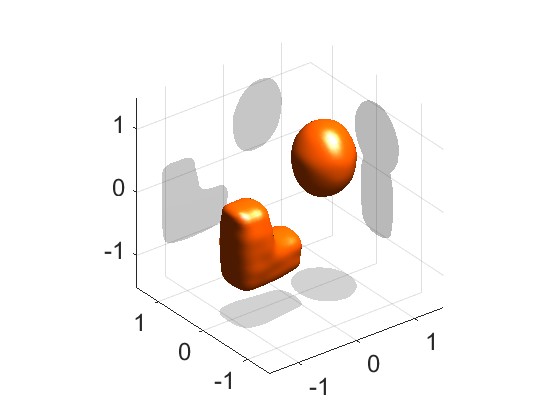}} \hspace{1.5cm}
\subfloat[$\mathcal{I}$ on $\{x_2 = 0\}$ for $k = 6$]{\includegraphics[width=5.5cm]{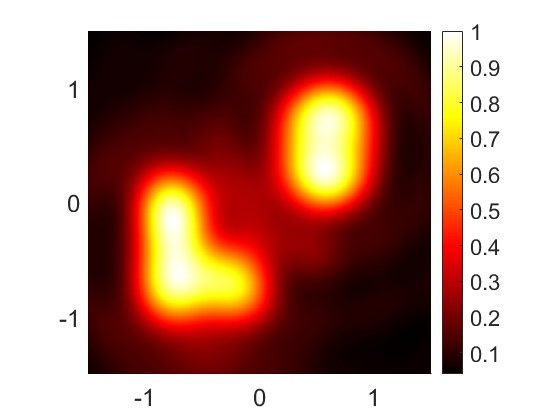}}  \hspace{0.75cm}
\subfloat[$\mathcal{I}$ on $\{x_2 = 0\}$ for $k  =12$ ]{\includegraphics[width=5.5cm]{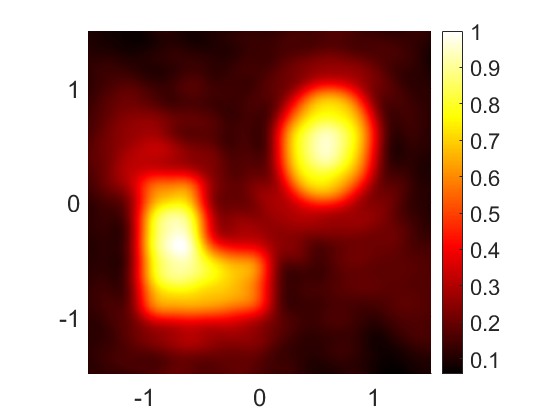}}  \hspace{0.75cm}

\caption{Reconstruction from   data associated with different wave numbers.  
The data is generated using 180 incident waves and
has a  2$\%$ noise level. $M = 3$ with $\p_1 = (1,0,0)^\top, \p_2 = (0,1,0)^\top, \p_3 = (0,0,1)^\top$.  
(a): true geometry. (b, d): reconstruction with $k  = 6$. (c, e): reconstruction with $k = 12$.
 } 
 \label{fi6}
\end{figure}

\subsubsection{Reconstruction with different numbers of incident waves (Figure \ref{fi4})}
We present the reconstruction results in Figure~\ref{fi4}  for noisy data using 48, 96 and 180 incident waves. The data is added with 
 20$\%$ noise and correspond to  the wave number $k  = 12$. We consider $M = 3$ with the vectors $\p_1 = (1,0,0)^\top, \p_2 = (0,1,0)^\top, \p_3 = (0,0,1)^\top$. 
 Remarkably, the reconstructions are already quite good when using 48 incident waves to generate the data. The results show a slight improvement with the use of more incident waves. Importantly, these results remain consistent even when more than 180 incident waves are employed in the reconstruction process.

\begin{figure}[h!!!]
\centering
\subfloat[True geometry]{\includegraphics[width=6.3cm]{casestairball_pset1_k_12_Nx_324_Nd_96_3Dview_true}} \hspace{-1cm}\\
\subfloat[$N = 48$]{\includegraphics[width=6.3cm]{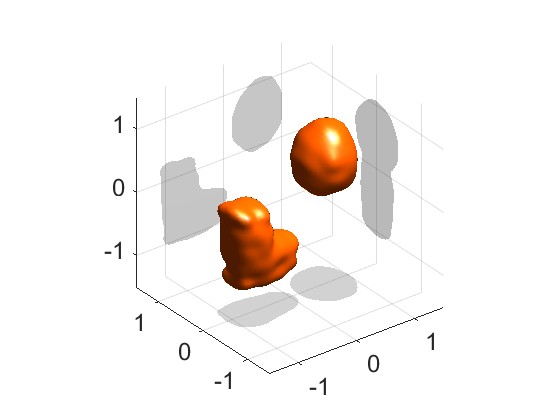}} \hspace{-1.45cm}
\subfloat[$N = 96$]{\includegraphics[width=6.3cm]{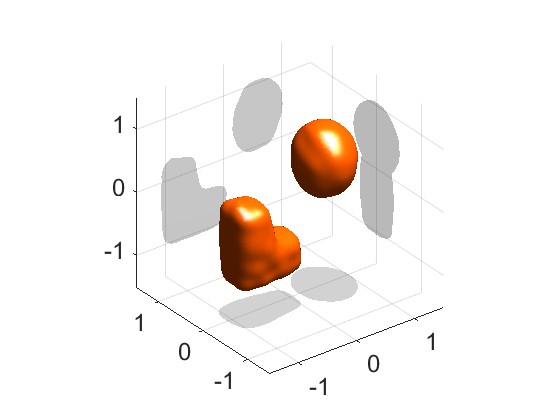}} \hspace{-1.45cm}
\subfloat[$N = 180$]{\includegraphics[width=6.3cm]{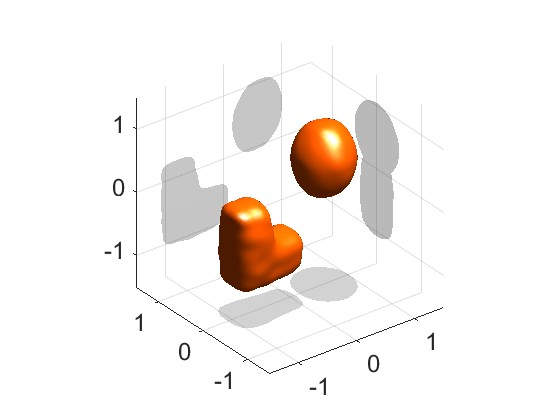}} \\
\subfloat[ $\mathcal{I}$ on $\{x_2 = 0\}$ for $N = 48$]{\includegraphics[width=4.5cm]{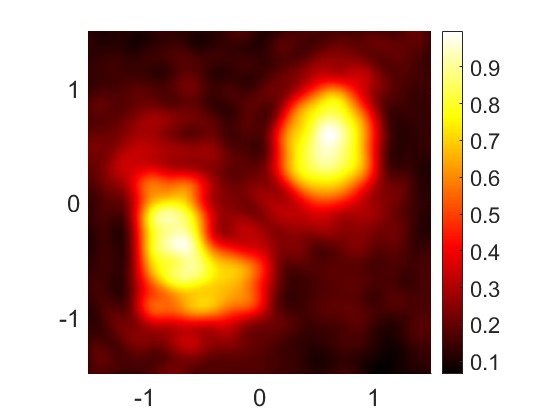}} \hspace{0.35cm}
\subfloat[ $\mathcal{I}$ on $\{x_2 = 0\}$ for $N = 96$]{\includegraphics[width=4.5cm]{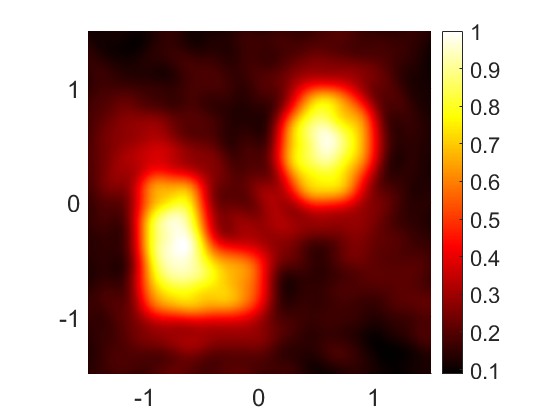}} \hspace{0.35cm}
\subfloat[ $\mathcal{I}$ on $\{x_2 = 0\}$ for $N = 180$]{\includegraphics[width=4.5cm]{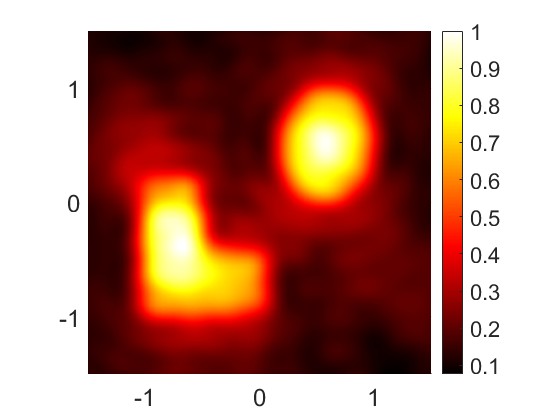}} 

\caption{Reconstruction  with  noisy data ($20\%$ noise) using $N$ incident waves, $k = 12$, and $M = 3$ with $\p_1 = (1,0,0)^\top, \p_2 = (0,1,0)^\top, \p_3 = (0,0,1)^\top$.  
First row (a): true geometry.  (b, e): reconstruction with $N = 48$. (c, f): reconstruction with $N = 96$.  (d, g): reconstruction with $N = 180$.
 } 
 \label{fi4}
\end{figure}

%

\subsubsection{Reconstruction  with different sets of  vectors $\p_n$ (Figure \ref{fi5})}

We present the reconstruction results in Figure~\ref{fi5}  for noisy data using different sets of vectors $\p_n$ for $n = 1, 2, \dots, M$. The data is added with  20$\%$ noise and correspond to  the wave number $k  = 12$. We consider three following cases: $M =1$ and  $\p_1 = (0,1,0)^\top$, $M =1$ and  $\p_1 = (1/\sqrt{2},0,1/\sqrt{2})^\top$, $M = 3,$ and  $\p_1 = (1/\sqrt{2},1/\sqrt{2},0)^\top, \p_2 = (0,1/\sqrt{2},1/\sqrt{2})^\top, \p_3 = (1/\sqrt{2},0,1/\sqrt{2})^\top$. For $M = 1$, the reconstructions using $\p_1 = (0,1,0)^\top$  and $\p_1 = (1/\sqrt{2},0,1/\sqrt{2})^\top$ 
yield noticeably different results and are not as satisfactory as those obtained with $M  =  3$. 
We also see that the result for $M = 3$ is similar to that of the case  $M = 3$ and $\p_1 = (1,0,0)^\top, \p_2 = (0,1,0)^\top, \p_3 = (0,0,1)^\top$ presented  in Figure~\ref{fi4}.  This aligns with the observation from Figure~\ref{fi00}, where using more $\p_n$  leads to improved peaks for the kernel function $\mathcal{K}$. 
Moreover,
our numerical study  supports the finding that employing $M = 3$ with three linearly independent vectors $\p_n$ appears to provide the most reasonable and stable results.

\begin{figure}[h!!!]
\centering
\subfloat[True geometry]{\includegraphics[width=6.3cm]{casestairball_pset1_k_12_Nx_324_Nd_96_3Dview_true}} \\
\subfloat[$M = 1$]{\includegraphics[width=6.3cm]{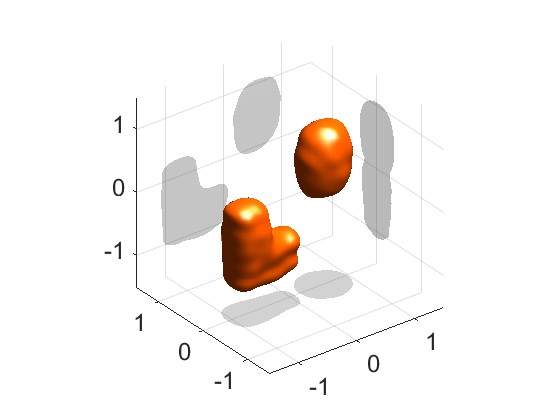}} \hspace{-1.45cm}
\subfloat[$M = 1$]{\includegraphics[width=6.3cm]{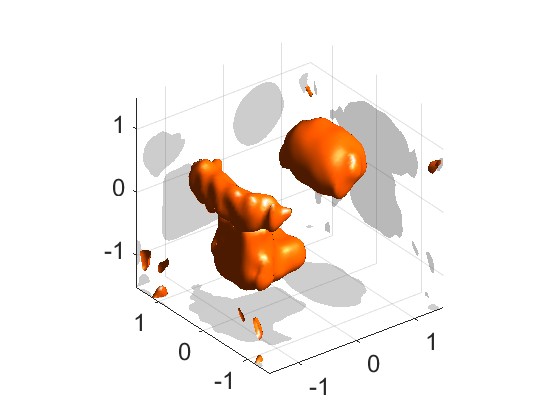}} \hspace{-1.45cm}
\subfloat[$M = 3$]{\includegraphics[width=6.3cm]{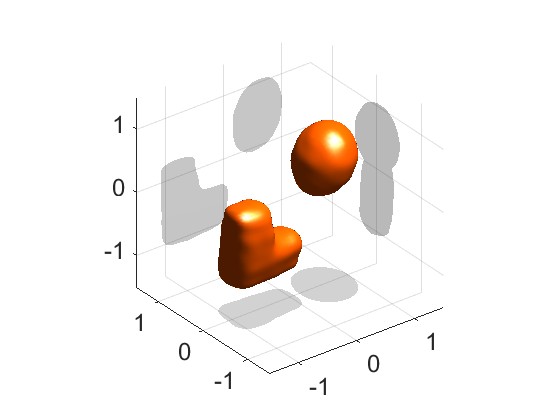}} \\
\subfloat[ $\mathcal{I}$ on $\{x_2 = 0\}$ for $M = 1$]{\includegraphics[width=4.5cm]{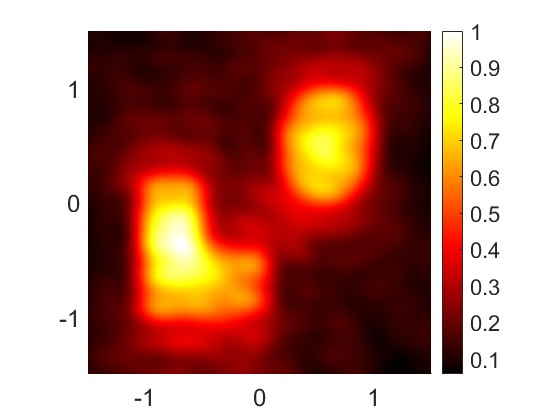}}  \hspace{0.75cm}
\subfloat[ $\mathcal{I}$ on $\{x_2 = 0\}$ for $M = 1$]{\includegraphics[width=4.5cm]{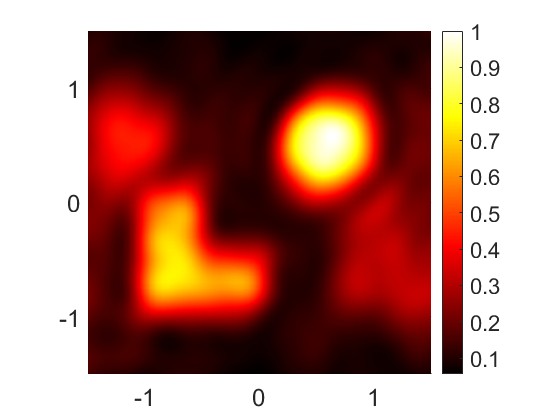}}  \hspace{0.75cm}
\subfloat[ $\mathcal{I}$ on $\{x_2 = 0\}$ for $M = 3$]{\includegraphics[width=4.5cm]{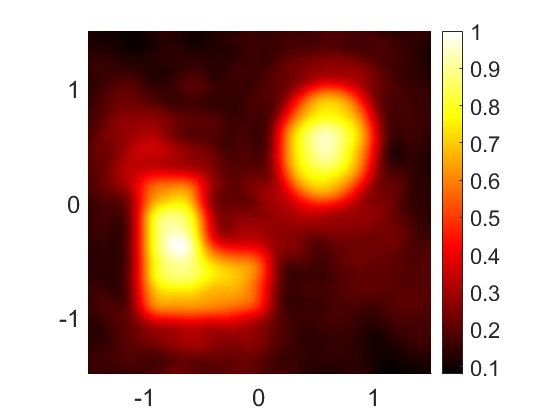}} 

\caption{Reconstruction  using different sets of vectors $\p_n$, $n = 1, 2, \dots, M$, $k = 12, N = 180$, and 20\% noise in the data.
(a): True geometry. (b, e): reconstruction with $M =1, \p_1 = (0,1,0)^\top$. (c, f): reconstruction with $M =1, \p_1 = (1/\sqrt{2},0,1/\sqrt{2})^\top$.
 (d, g): reconstruction with $M = 3, \p_1 = (1/\sqrt{2},1/\sqrt{2},0)^\top, \p_2 = (0,1/\sqrt{2},1/\sqrt{2})^\top, \p_3 = (1/\sqrt{2},0,1/\sqrt{2})^\top$.
 } 
 \label{fi5}
\end{figure}

%

\vspace{0.5cm}
\textbf{Acknowledgement}. The first author acknowledges Ho Chi
Minh City University of Technology and Education for funding her work through the
grant No T2023-71.
The work of  the second author  was partially supported by NSF grant DMS-2208293.

\bibliographystyle{plain}
\bibliography{ip-biblio}

\end{document}